\documentclass[a4paper,11pt,twoside, eqno]{article}
\usepackage{mathrsfs}
\usepackage{amssymb}
\usepackage{amsmath}
\usepackage{mathrsfs}
\usepackage{amsthm}
\usepackage{amsfonts}
\usepackage{color}
\usepackage{latexsym}
\usepackage{txfonts}
\usepackage{indentfirst}
\usepackage{soul}

\usepackage{amscd}
\usepackage{bbm}
\usepackage{graphicx}
\usepackage[all]{xy}
\usepackage[mathscr]{eucal}
\usepackage{cases}
\usepackage{psfrag}
\usepackage{subfigure}

\newcommand{\R}{{\mathbb R}}

\allowdisplaybreaks
\newtheorem{theorem}{Theorem}[section]

\newtheorem{corollary}[theorem]{Corollary}

\newtheorem{definition}[theorem]{Definition}

\newtheorem{remark}[theorem]{Remark}

\newtheorem{lemma}[theorem]{Lemma}

\newtheorem{proposition}[theorem]{Proposition}

\numberwithin{equation}{section}

\begin{document}

\title{\bf\Large The Morse index theorem in the case of two variable\\
endpoints in conic Finsler manifolds
\footnotetext{\hspace{-0.35cm} 2020
{\it Mathematics Subject Classification}.
Primary 53C22, 53C50, 53C60.
\endgraf
{\it Key words and phrases.}
pseudo-Finsler manifold, Morse index form, Jacobi fields, focal points.
}}
\date{}
\author{Guangcun Lu\footnote{
E-mail: \texttt{gclu@bnu.edu.cn}/{February 2, 2023(first), 
July 28, 2023(third revised)}.}}

\maketitle

\vspace{-0.7cm}

\begin{center}
\begin{minipage}{13cm}
{\small {\bf Abstract}\quad
In this note,  we prove the Morse index theorem for a geodesic connecting
two submanifolds in a $C^7$ manifold with a $C^6$ (conic) pseudo-Finsler metric
provided that the fundamental tensor is positive definite along velocity curve of the geodesic.
}
\end{minipage}
\end{center}
	
\vspace{0.2cm}

\section{Introduction, preliminaries and statements of results}\label{sec:Fin}

\noindent{\bf 1.1. Introduction}.
Ioan Radu Peter proved the Morse index theorem where the ends are submanifolds in Finsler geometry in \cite{Pe06}.
In this paper we generalize it to (conic) pseudo-Finsler
manifolds with lower smoothness.
Instead of the Cartan connection used in \cite{Pe06}
we employ   the Chern connection to introduce the Morse index form, $P$-Jacobi field and the normal second fundamental form
(or shape operator).
In addition, our proof method is different from that of \cite{Pe06}. For example,
we do not use the (conic) pseudo-Finsler version (Proposition~\ref{prop:Sa2.10}) of the so-called index lemma \cite[Lemma~4.2]{Pe06}
in Finsler geometry. The Morse index theorem proved in
\cite[\S31]{Her68} by R.A. Hermann plays a key role in our proof.\\

\noindent{\bf 1.2. Preliminaries for Finsler geometry}.
 The theory of geodesics and Jacobi fields on smooth pseudo-Finsler manifolds was developed by
 Javaloyes and collaborators in \cite{HubJa22, Jav13, Jav13+, Jav14, Jav15, Jav20, Jav21}.
  The related notions and results may be stated for pseudo-Finsler manifolds with lower smoothness.

Let $M$ be a $n$-dimensional, connected $C^7$ submanifold of $\R^N$.
Its tangent bundle $TM$ is a $C^6$ manifold of dimension $2n$.
Let $\pi:TM\rightarrow M$ be the natural projection and
let $A\subset TM\setminus 0_{TM}$ be an open subset of $TM$ which is conic, that is,
such that $\pi(A)=M$ and $\lambda v\in A$, for every $v\in A$ and $\lambda>0$.
A $C^6$  (conic) \textsf{pseudo-Finsler metric} (with domain $A$) is a $C^6$ function
$L:A\rightarrow \R$ satisfying the following conditions:
\begin{description}
\item[(i)]  $L(\lambda v)=\lambda^2 L(v)$ for any $v\in A$ and $\lambda>0$,
\item[(ii)] for every $v\in A$,  the \textsf{fundamental tensor} $g_v$ of $L$ at
$v$ defined by
$$
T_{\pi(v)}M\times T_{\pi(v)}M\ni (u,v)\mapsto
g_v(u,w):=\frac 12 \frac{\partial^2}{\partial t\partial s} L(v+tu+sw)|_{t=s=0}
$$
 is nondegenerate.
\end{description}
The pair $(M,L)$ is called a $C^6$ (conic) \textsf{pseudo-Finsler manifold} (with domain $A$).
In particular, if each $g_v$ in (ii) is positive definite, we call the square root $F=\sqrt{L}$
 a $C^6$ \textsf{conic Finsler metric} (cf. \cite{Jav14}).
By \cite[Proposition~2.3]{Jav15} $g_v(v,v)=L(v)$ and $g_{\lambda v}=g_v$ for any $v\in A$ and
$\lambda>0$.
The  \textsf{Cartan tensor} of $L$ is  the trilinear form
\begin{equation}\label{cartantensor}
C_v(w_1,w_2,w_3)=\frac 14\left. \frac{\partial^3}{\partial s_3\partial s_2\partial s_1}L\left(v+\sum_{i=1}^3 s_iw_i\right)\right|_{s_1=s_2=s_3=0}
\end{equation}
for $v\in A$ and $w_1,w_2,w_3\in T_{\pi(v)} M$. By homogeneity, $C_v(w_1,w_2,w_3)=0$ if $w_i=v$ for some $i$.

For an open subset $\Omega\subset M$ and an integer $0\le\ell\le 6$
let $\Gamma^\ell(T\Omega)$ be the space of $C^\ell$ vector fields on $\Omega$.
$V\in \Gamma^\ell(T\Omega)$ is called \textsf{$A$-admissible} if $V(x)\in A$ for every $x\in\Omega$.
Denote by $\Gamma^\ell_A(T\Omega)$ the set of all $A$-admissible $V\in \Gamma^\ell(T\Omega)$.

Let $(\Omega, x^i)$ be a  coordinate chart on $M$ and
 let $(x^i,y^i)$ be  the canonical coordinates  on $\pi^{-1}(\Omega)\subset TM$ associated with it, i.e.,
 $v=v^i\frac{\partial}{\partial x^i}\big|_{\pi(v)}=y^i(v)\frac{\partial}{\partial x^i}\big|_{\pi(v)}$
 for $v\in\pi^{-1}(\Omega)$.
 Define
\begin{eqnarray*}
&&T\Omega\cap A\ni v\mapsto g_{ij}(v)=g_v\left(\frac{\partial}{\partial x^i}\Big|_{\pi(v)}, \frac{\partial}{\partial x^j}\Big|_{\pi(v)}\right),\\
&&T\Omega\cap A\ni v\mapsto C_{ijk}(v)=C_v\left(\frac{\partial}{\partial x^i}\Big|_{\pi(v)}, \frac{\partial}{\partial x^j}\Big|_{\pi(v)}, \frac{\partial}{\partial x^k}\Big|_{\pi(v)}\right).
\end{eqnarray*}
They are $C^4$ and $C^3$ functions, respectively.
For $v\in T\Omega\cap A$ let $(g^{ij}(v))=(g_{ij}(v))^{-1}$ and
\begin{eqnarray*}
&&\gamma^i_{jm}(v):=\frac{1}{2}\sum_s{g}^{is}(v)\left(\frac{\partial{g}_{sj}}{\partial{x^m}}(v)+
\frac{\partial{g}_{ms}}{\partial{x^j}}(v)-\frac{\partial{g}_{jm}}{\partial{x^s}}(v)\right),\\
&&{C}^i_{jm}(v):= \sum_l{g}^{il}(v){C}_{ljm}(v),\\
&&N^i_j(v):=\sum_m\gamma^i_{jm}(x,y)v^m-\sum_{m,r,s}\mathcal{C}^i_{jm}(v)\gamma^m_{rs}(v)v^rv^s,\\
&&\Gamma^i_{jk}(v):=\gamma^i_{jk}(v)-\sum_{l,r}g^{li}(v)\left(C_{jlr}(v)N^r_k(v)-C_{jkr}(v)N^r_l(v)+ C_{lkr}(v)N^r_j(v)\right).
\end{eqnarray*}
These are $C^3$. Moreover
 $N^j_i(v)=\sum_k\Gamma^j_{ik}(v)v^k$ and the Chern connection  $\nabla$
 (on the pulled-back tangent bundle $\pi^\ast TM$) is given by
\begin{equation}\label{e:inducedConn-}
\nabla_{\partial_{x^i}}\partial_{x^j}(v)=\sum_m\Gamma^i_{jm}(v)\partial_{x^m},\quad i,j\in\{1,\cdots,n\}.
\end{equation}
 The trilinear map $R_v$ from $T_{\pi(v)}M\times T_{\pi(v)}M\times T_{\pi(v)}M$ to $T_{\pi(v)}M$ given by
\begin{equation}\label{e:ChernC}
R_v(\xi,\eta)\zeta=\sum_{i,j,k,l}\xi^k \eta^l \zeta^j R_{j\,\,kl}^{\,\,\,i}(v) \partial_{x^i}|_{\pi(v)}
\end{equation}
for $v\in T\Omega\cap (TM\setminus 0_{TM})$ defines  the \textsf{Chern curvature tensor} (or
\textsf{$hh$-curvature tensor} \cite[(3.3.2) \& Exercise 3.9.6]{BaChSh})  $R_V$
on $\Omega\subset M$, where
$$
R^i_{jkl}=\frac{\delta\Gamma^i_{jl}}{\delta x^k}-\frac{\delta\Gamma^i_{jk}}{\delta x^l}+
\sum_h\Gamma^i_{hk}\Gamma^h_{jl}-\sum_h\Gamma^i_{hl}\Gamma^h_{jk}\quad\hbox{and}\quad
\frac{\delta}{\delta x^k}=\frac{\partial}{\partial x^k}-\sum_iN^i_k\frac{\partial}{\partial y^i}.
$$

Let $0\le \ell\le 3$, $s\ge 1$, $V\in \Gamma^\ell_A(T\Omega)$. Define $\nabla^V:\Gamma^s(T\Omega)\times\Gamma^s(T\Omega)\to \Gamma^{\min\{\ell,s-1\}}(T\Omega)$  by
\begin{equation}\label{e:inducedConn}
\nabla^V_{f\partial_{x^i}}(g\partial_{x^j})(x)=f(x)g(x)\sum_m\Gamma^i_{jm}(V(x))\partial_{x^m}+ f(x)\partial_{x^i}g(x)\partial_{x^j}
\end{equation}
for $f,g\in C^s(\Omega)$ and $i,j\in\{1,\cdots,n\}$.
(This depends only on $V(x)$ as in the proof of \cite[Lemma~3.12]{Rad04}.)
It satisfies the following properties for $f\in C^s(\Omega)$ and $X,Y,Z\in\Gamma^s(T\Omega)$:
\begin{description}
\item[(i)]  $\nabla^V_X(fY)=f\nabla^V_XY+ X(f)Y$ and
$\nabla^V_{fX}Y=f\nabla^V_XY$,
\item[(ii)] $\nabla^V_XY-\nabla^V_YX=[X,Y]$,
\item[(iii)] $X(g_V(Y,Z))=g_V(\nabla^V_XY, Z)+ g_V(Y,\nabla^V_XZ)+ 2C_V(\nabla^V_XY,Y,Z)$,
where $g_V$ and $C_V$ are, respectively, the fundamental tensor and the Cartan tensor of $L$  evaluated on the vector field $V$.
\end{description}
 $\nabla^V$ is called the  \textsf{Chern connection} of  $(M,L)$  associated with $V$.
For a $C^1$ function $f:A\to\mathbb{R}$, the \textsf{vertical derivative} of $f$ at $v\in A$ in direction $u\in TM$ with $\pi(u)=\pi(v)$
was defined in \cite{Jav21} by
$$
\partial^\nu f_v(u)=\left.\frac{d}{dt}\right|_{t=0}f(v+tu).
$$
 Define the \textsf{vertical gradient} $(\nabla^\nu f)_v$ at $v\in A$, and the \textsf{horizontal gradient} $(\nabla^{\rm h} f)_v$ at $v\in A$ by \begin{eqnarray}\label{Vnabla}
\partial^\nu f_v(u)=g_v((\nabla^\nu f)_v,u)\;\;\forall u\in\pi^{-1}(\pi(v)),\\
u(f(V))-\partial^\nu f_v(\nabla^V_uV)=g_v((\nabla^{\rm h} f)_v,u)\;\;\forall u\in\pi^{-1}(\pi(v)),\label{Hnabla}
\end{eqnarray}
respectively, where $V$ is a local $A$-admissible  $C^1$ extension of $v$. See \cite[(3),(4),(5)]{Jav21}.

If $1\le\ell\le 3$ and $s\ge 2$,  the connection $\nabla^V$ in (\ref{e:inducedConn}) has the curvature tensor $R^V$ defined by
$$
R^V(X,Y)Z=\nabla^V_X\nabla^V_YZ-\nabla^V_Y\nabla^V_XZ-
\nabla^V_{[X,Y]}Z
$$
for  vector fields $V\in \Gamma^\ell_A(T\Omega)$ and  $X,Y,Z\in\Gamma^s(T\Omega)$.
Clearly, $R^V(X,Y)Z$ is $C^{\min\{\ell-1,s-2\}}$,  and it was proved in \cite[Theorem~2.1]{Jav13+}
that
\begin{equation}\label{e:ChernCurv}
R_V(X,Y)Z=R^V(X,Y)Z-P_V(Y,Z,\nabla^V_XV)+ P_V(X,Z, \nabla^V_YV),
\end{equation}
where $P_V(X,Y,Z)=\frac{\partial}{\partial t}(\nabla_X^{V+tZ})|_{t=0}$
and so
$P_V(X,Y,V)=0$ (\cite[(1)]{Jav13+}). The Chern tensor $P_V(X,Y,Z)$ is symmetric in $X$ and $Y$, and
$P_V(X,Y,Z)|_x$ depends only on $V(x)$. By \cite[Lemma~1.2]{Jav13+}
\begin{equation}\label{e:P-tensor}
P_v(v,v,u)=0,\quad\forall v\in T\Omega\cap A,\;\forall u\in T_{\pi(v)}M.
\end{equation}
 For $v\in (T_xM)\cap A$ and $u,w\in T_xM$, by \cite[Proposition~3.1]{Jav20} we have
\begin{eqnarray}\label{e:Rsym7}
g_{v}(R_{v}(v, u)v, w)&=&g_{v}(R_{v}(v, w)v, u),\\
g_{v}(R_{v}(u, v)v, w)&=&
-g_{v}(R_{v}(u,v)w, v).\label{e:Rsym8}
\end{eqnarray}

 For a $C^7$ connected submanifold $P\subset M$ of dimension $k<n$, call
\begin{eqnarray}\label{e:normalBundle}
TP^\bot=\{v\in A~|~\pi(v)\in P,\;g_v(v,w)=0\;\forall w\in T_{\pi(v)}P\}
\end{eqnarray}
 the  \textsf{normal bundle} of $P$ in $(M,L)$
though $TP^\bot$ is not necessarily a fiber bundle over $P$.
 By \cite[Lemma~3.3]{Jav15} a nonempty $TP^\perp$ is an $n$-dimensional submanifold
of $TM$, $P_0=\pi(TP^\bot)$ is open in $P$, and the map
$\pi:TP^\perp\rightarrow P_0$ is a submersion.
For $v\in TP^\bot$ with $\pi(v)=p$, suppose that $g_v|_{T_pP\times T_pP}$ is nondegenerate.
(This clearly holds if $P$ is a point.) Then there exists a splitting
\begin{equation}\label{e:splitting}
T_pM=T_pP\oplus (T_pP)^\perp_v,
\end{equation}
where $(T_pP)^\perp_v=\{u\in T_pM\,|\,g_v(u,w)=0\;\forall w\in T_{p}P\}$.
Clearly, $v\in (T_pP)^\perp_v$, and each $u\in T_pM$ has a decomposition ${\rm tan}^P_v(u)+{\rm nor}^P_v(u)$,
where ${\rm tan}^P_v(u)\in T_pP$ and ${\rm nor}^P_v(u)\in (T_pP)^\perp_v$.
Let $V$ be local $C^6$ admissible extensions of $v$, and for $u,w\in T_pP$ let $U$ and $W$ be local $C^6$ extensions of
 $u$ and $w$ to $M$ in such a way that $U$ and $W$ are tangent to $P$ along $P$.
Then
 \begin{eqnarray*}
 &&S^P_{v}:T_pP\times T_pP\to (T_pP)^\perp_{v}, \;(u,w)\mapsto  {\rm nor}^P_v\left((\nabla^V_UW)(p)\right),\\
 &&\tilde{S}^P_{v}:T_pP\to T_pP,\;u\mapsto {\rm tan}^P_v\left((\nabla^V_UV)(p)\right)
 \end{eqnarray*}
are well-defined (cf. \cite[subsection~3.1]{Jav15} and \cite[\S3.1]{HubJa22}).
They are called  the \textsf{second fundamental form} of $P$ in the direction $v$ and
 the \textsf{normal second fundamental form} (or \textsf{shape operator}) of $P$ in the direction $v$, respectively.
$S^P_{v}$ is bilinear and symmetric, $\tilde{S}^P_{v}$ is linear, and
$g_{v}({S}^P_{v}(u,w), v)=-g_{v}(\tilde{S}^P_{v}(u),w)$. Hence $\tilde{S}^P_{v}$ is symmetric.

For a curve $c\in W^{1,2}([a, b], M)$  let
$W^{1,2}(c^\ast TM)$ and $L^2(c^\ast TM)$ denote the spaces of all   $W^{1,2}$ and  $L^{2}$ vector fields along $c$, respectively.
Then  $\dot{c}\in L^2(c^\ast TM)$. Let
$(x^i, y^i)$ be the canonical coordinates around $\dot{c}(t)\in TM$.
Write
$\dot c(t)=\dot{c}^i(t) \partial_{x^i}|_{c(t)}$ and
$\zeta(t)=\zeta^i(t)\partial_{x^i}|_{c(t)}$
for a vector field $\zeta$ along $c$.
Call $\xi\in C^0(c^\ast TM)$ \textsf{$A$-admissible} if $\xi(t)\in A$ for all $t\in [a,b]$.
 The \textsf{covariant derivative} of $\zeta\in W^{1,2}(c^\ast TM)$ along  $c$ (with this $A$-admissible $\xi$ as reference vector)  is defined by
\begin{equation}\label{e:covariant-derivative}
D^\xi_{\dot{c}}\zeta(t):= \sum_m\bigl(\dot{\zeta}^m(t)
+ \sum_{i,j} \zeta^i(t)\dot{c}^j(t)\Gamma_{ij}^m(c(t), \xi(t))\bigr)\partial_{x^m}|_{c(t)}.
\end{equation}
Clearly, $D^\xi_{\dot{c}}$ belongs to $L^2(c^\ast TM)$, and sits in $C^{\min\{1,r\}}(c^\ast TM)$ provided that
 $c$ is of class $C^{r+1}$, $\zeta\in C^{r+1}(c^\ast TM)$ and $\xi\in C^r(c^\ast TM)$ for some $0\le r\le 6$;  $D^\xi_{\dot{c}}\zeta(t)$ depends only on $\xi(t)$, $\dot{c}(t)$ and behavior of $\zeta$ near $t$. Moreover $D^\xi_{\dot{c}}\zeta(t)=\nabla^{\tilde\xi}_{\dot{c}}\tilde\zeta(c(t))$ if $\dot{c}(t)\ne 0$
 and $\tilde\xi$ (resp. $\tilde\zeta$) is any $A$-admissible extension (resp. any extension) of $\xi$ (resp. $\zeta$) near $c(t)$.
When the above $\xi$ belongs to $W^{1,2}(c^\ast TM)$, $D^\xi_{\dot{c}}$  is \textsf{ almost $g_\xi$-compatible}, that is, for any $\eta,\zeta\in W^{1, 2}(c^\ast TM)$ we have
\begin{equation}\label{e:compat}
\frac{d}{dt}g_{\xi}(\zeta,\eta)=g_{\xi}\bigl(D^\xi_{\dot{c}}\zeta,\eta\bigr)
+g_{\xi}\bigl(\zeta, D^\xi_{\dot{c}}\eta\bigr)
+2{C}_{\xi}\bigl(D^\xi_{\dot{c}}\xi,\zeta,\eta\bigr)\quad{\rm a.e.}
\end{equation}
(cf. \cite[(4)]{Jav13}). Moreover, if $c$, $\xi$ and $\zeta$ are $C^3$, $C^1$ and $C^2$, respectively,
then $D^\xi_{\dot{c}}\zeta$ is $C^1$ and $D^\xi_{\dot{c}}D^\xi_{\dot{c}}\zeta$ is well-defined and is $C^0$.
A $C^1$ curve $\gamma:[a,b]\to M$ is said to be \textsf{$A$-admissible}
if $\dot\gamma\in C^0(\gamma^\ast TM)$ is $A$-admissible.
A $W^{1,2}$-vector field $X$ along such a curve $\gamma$ is called \textsf{parallel} if $D^{\dot\gamma}_{\dot\gamma}X=0$.
 A $C^2$ $A$-admissible curve $\gamma$ in $(M, L)$
 is called  an  \textsf{$L$-geodesic} if $\dot\gamma$ is parallel along $\gamma$, i.e.,
 $D^{\dot \gamma}_{\dot{\gamma}}\dot{\gamma}(t)\equiv 0$.
$L$ is always constant along a geodesic $\gamma$. In particular, if $L(\dot\gamma(t))\equiv 0$
then $\gamma$ is called a \textsf{lightlike geodesic}. \\

\noindent{\bf Remark~A}.
$L$-geodesics must first be $A$-admissible, and so constant curves in $(M, L)$ cannot be geodesics.
For an $L$-geodesic $\gamma$ in $M$,  each $\dot\gamma(t)$ belongs to $A$ and therefore $g_{\dot\gamma(t)}$ must be nondegenerate.
By \cite[Proposition~2.3]{Jav15} $g_{\dot\gamma(t)}(\dot\gamma(t),\dot\gamma(t))=L(\dot\gamma(t))\equiv L(\dot\gamma(0))$.
If $g_{\dot\gamma(t_0)}$ is positive definite for some $t_0\in [0, \tau]$, so is each $g_{\dot\gamma(t)}$
because $t\mapsto g_{\dot\gamma(t)}$ is continuous;
see the arguments below (\ref{e:secondDifff2}).\\

Observe that in local coordinates the geodesic equation has the form
\begin{equation*}
 \ddot\gamma^k+\sum_{i,j=1}^n\dot\gamma^i \dot\gamma^j(\Gamma^k_{ij}\circ\dot\gamma)=0,\quad k\in\{1,\dots,n\}.
\end{equation*}
Since $\Gamma^k_{ij}$ is $C^3$, the geodesics are actually $C^4$ by the ordinary differential equation theory.

As usual, using  a $C^6$ Riemannian metric $h$ on $M$ and its exponential map
 $\exp$ may determine a Riemannian-Hilbert metric
on $W^{1,2}([0,\tau]; M)$ given by
\begin{equation}\label{e:1.1}
\langle\xi,\eta\rangle_1=\int^\tau_0 h(\xi(t),\eta(t)) dt+
\int^\tau_0 h(\nabla^h_{\dot\gamma}\xi(t),\nabla^h_{\dot\gamma}\xi(t)) dt
\end{equation}
 for $\xi,\eta\in T_\gamma W^{1,2}([0,\tau]; M)=W^{1,2}(\gamma^\ast TM)$
(where the $L^2$ covariant derivative along $\gamma$ associated with the Levi-Civita connection $\nabla^h$ of the metric $h$
is defined as (\ref{e:covariant-derivative})).

Let $P$ be as above, and let $Q$ be another $C^7$ connected submanifold in $M$ of dimension less than $n$. Consider the submanifold of
$W^{1,2}([0,\tau]; M)$,
$$
W^{1,2}([0,\tau]; M,P,Q):=\{\gamma\in W^{1,2}([0,\tau]; M)\,|\,\gamma(0)\in P,\,\gamma(\tau)\in Q\}.
$$
Its tangent space at $\gamma\in W^{1,2}([0,\tau]; M,P,Q)$ is
$$
W^{1,2}_{P\times Q}(\gamma^\ast M):=\{\xi\in W^{1,2}(\gamma^\ast TM)\,|\,\xi(0)\in T_{\gamma(0)}P,\;\xi(\tau)\in T_{\gamma(\tau)}Q\}.
$$
Let $C^1([0,\tau];M, P, Q)=\{\gamma\in C^1([0,\tau];M)\,|\, \gamma(0)\in P,\;\gamma(\tau)\in Q\}$ and let
$$
C^1_A([0,\tau];M, P, Q)
$$
consist of all $A$-admissible  curves in $C^1([0,\tau];M, P, Q)$. The latter is
 an open subset of the Banach manifold $C^1([0,\tau];M, P, Q)$.
Define the energy functional $\mathcal{E}_{P,Q}:C^1_A([0,\tau];M, P, Q)\to\R$ by
\begin{equation}\label{e:Fenergy}
\mathcal{E}_{P,Q}(\gamma)=\frac{1}{2}\int^\tau_0L(\dot{\gamma}(t))dt.
\end{equation}
It is $C^{2}$, and  $\gamma\in C^1_A([0,\tau];M, P, Q)$ is a critical point of $\mathcal{E}_{P,Q}$ if and only if $\gamma$
is an $L$-geodesic satisfying the boundary condition
\begin{equation}\label{e:1.4}
\left\{\begin{array}{ll}
&g_{\dot\gamma(0)}(u,\dot\gamma(0))=0\quad\forall u\in
T_{\gamma(0)}P,\\
&g_{\dot\gamma(\tau)}(v,\dot\gamma(\tau))=0\quad\forall v\in
T_{\gamma(\tau)}Q \end{array}\right.
\end{equation}
(cf. \cite[Chap.1, \S1]{BuGiHi}, \cite[Proposition~2.1]{CaJaMa3} and \cite[Prop.~3.1, Cor.3.7]{Jav15}.
Note that $dL(\dot\gamma(t))\equiv 0$).
The curve $\gamma$ satisfying (\ref{e:1.4})
is said to be $g_{\dot\gamma}$-orthogonal (or perpendicular) to $P$ and $Q$.
Clearly, (\ref{e:1.4}) implies that $\dot\gamma(0)\in TP^\bot\subset A$ and
 $\dot\gamma(\tau)\in TQ^\bot\subset A$.

For an $L$-geodesic $\gamma:[0,\tau]\to M$ satisfying (\ref{e:1.4}), from now on
 we also suppose that  both $g_{\dot\gamma(0)}|_{T_{\gamma(0)}P\times T_{\gamma(0)}P}$ and
 $g_{\dot\gamma(\tau)}|_{T_{\gamma(\tau)}Q\times T_{\gamma(\tau)}Q}$ are nondegenerate.
 (As noted in \cite[Remark~3.9]{Jav15}, these hold if $P$ (resp. $Q$) is a hypersurface of $M$
 and $L(\dot\gamma(0))\ne 0$ (resp. $L(\dot\gamma(\tau))\ne 0$).)
 Then  the normal second fundamental forms $\tilde{S}^P_{\dot\gamma(0)}$ and $\tilde{S}^Q_{\dot\gamma(\tau)}$ are well-defined.
 Suppose that $\gamma\in C^1_A([0,\tau];M, P, Q)$ is a geodesic as above, hence $C^4$.
 Let $\Lambda:[0,\tau]\times(-\varepsilon,\varepsilon)\to M,\;(t,s)\mapsto\gamma_s(t)$ be a $C^4$ variation of $\gamma_0=\gamma$
 and $W=\partial_s\Lambda|_{s=0}$.
 By \cite[\S3]{Jav15} we have the following relation between
 the Hessian of $\mathcal{E}_{P,Q}$ at  $\gamma\in C^1_A([0,\tau];M, P, Q)$ and the second variation $\frac{d^2}{ds^2}\mathcal{E}_{P,Q}(\gamma_s)|_{s=0}$:
\begin{eqnarray}\label{e:secondDiff}
D^2\mathcal{E}_{P,Q}(\gamma)[W, W]&=&\frac{d^2}{ds^2}\mathcal{E}_{P,Q}(\gamma_s)|_{s=0}=\int_0^\tau \left(g_{\dot\gamma}(R_{\dot\gamma}(\dot\gamma,W)\dot\gamma, W)+g_{\dot\gamma}(D_{\dot\gamma}^{\dot\gamma}W,D_{\dot\gamma}^{\dot\gamma}W)\right)dt
\nonumber\\
&&+g_{\dot\gamma(0)}(\tilde{S}^P_{\dot\gamma(0)}(W(0)),W(0))
-g_{\dot\gamma(\tau)}(\tilde{S}^Q_{\dot\gamma(\tau)}(W(\tau)),W(\tau)).
 \end{eqnarray}
(Here we use the equality $R^{\gamma}(\dot\gamma,V)\dot\gamma=R_{\dot\gamma}(\dot\gamma,V)\dot\gamma$
 in \cite[page 66]{Jav15}.) If $V\in T_{\gamma}C^1_A([0,\tau];M, P, Q)$ comes from another $C^4$ variation of $\gamma_0=\gamma$,
as usual we derive from (\ref{e:Rsym7}) and symmetry of $\tilde{S}^P_{v}$ that
 \begin{eqnarray}\label{e:secondDiff*}
&&D^2\mathcal{E}_{P,Q}(\gamma)[V, W]=\int_0^\tau \left(g_{\dot\gamma}(R_{\dot\gamma}(\dot\gamma,V)\dot\gamma, W)+g_{\dot\gamma}(D_{\dot\gamma}^{\dot\gamma}V,D_{\dot\gamma}^{\dot\gamma}W)\right)dt
\nonumber\\
&&\qquad+g_{\dot\gamma(0)}(\tilde{S}^P_{\dot\gamma(0)}(V(0)),W(0))
-g_{\dot\gamma(\tau)}(\tilde{S}^Q_{\dot\gamma(\tau)}(V(\tau)),W(\tau)).
 \end{eqnarray}
 According to (\ref{e:covariant-derivative}),  it is clear that the right side of (\ref{e:secondDiff*})
actually defines  a continuous  symmetric bilinear form  ${\bf I}^\gamma_{P,Q}$ on $W^{1,2}_{P\times Q}(\gamma^\ast TM)$,
 called as the  \textsf{$(P,Q)$-index form of $\gamma$}.
 When $V$ is also $C^2$, by the proof of \cite[Proposition~3.11]{Jav15}
 \begin{eqnarray}\label{e:secondDiff+}
{\bf I}^\gamma_{P,Q}(V, W)&=&\int_0^\tau \left(g_{\dot\gamma}(R_{\dot\gamma}(\dot\gamma,V)\dot\gamma, W)-g_{\dot\gamma}(D_{\dot\gamma}^{\dot\gamma}D_{\dot\gamma}^{\dot\gamma}V,W)\right)dt
\nonumber\\
&&g_{\dot\gamma(0)}(D_{\dot\gamma}^{\dot\gamma}V(0)-\tilde{S}^P_{\dot\gamma(0)}(V(0)),W(0))\nonumber\\
&&-g_{\dot\gamma(\tau)}(D_{\dot\gamma}^{\dot\gamma}V(\tau)-\tilde{S}^Q_{\dot\gamma(\tau)}(V(\tau)),W(\tau)).
 \end{eqnarray}
 Then  $V\in W^{1,2}_{P\times Q}(\gamma^\ast TM)$ belongs to ${\rm Ker}({\bf I}^\gamma_{P,Q})$ if and only if
it is $C^4$ and satisfies
\begin{eqnarray} \label{e:kernel}
\left.\begin{array}{ll}
D_{\dot\gamma}^{\dot\gamma}D_{\dot\gamma}^{\dot\gamma}V-R_{\dot\gamma}(\dot\gamma,V)\dot\gamma=0,\\
{\rm tan}^P_{\dot\gamma(0)}\big((D_{\dot\gamma}^{\dot\gamma}V)(0)\big)=\tilde{S}^P_{\dot\gamma(0)}(V(0)),\\
{\rm tan}^Q_{\dot\gamma}\big((D_{\dot\gamma}^{\dot\gamma}V)(\tau)\big)=\tilde{S}^Q_{\dot\gamma(\tau)}(V(\tau))
\end{array}\right\}
 \end{eqnarray}
because $\gamma$ is $C^4$ and  $R^i_{jkl}$ are $C^{2}$.
To prove the implication to the right,   we may construct
 $C^4$ parallel orthonormal frame  fields $E_1,\cdots,E_n$ along $\gamma$ as done below (2.4),
and reduce the expected conclusions to the standard result in differential equations.
 Elements in $\mathscr{J}^{P,Q}_\gamma:={\rm Ker}({\bf I}^\gamma_{P,Q})$
are called \textsf{$(P,Q)$-Jacobi field} along $\gamma$.

Recall that $g_v(v,v)=L(v)$ (\cite[Proposition~2.3]{Jav15}).
Suppose $L\circ\dot\gamma\ne 0$. Then there exists a $\gamma^\ast TM=\mathbb{R}\dot{\gamma}\oplus\dot\gamma^\bot$, where
 $(\dot\gamma^\bot)_t=\{v\in T_{\gamma(t)}M\,|\, g_{\dot{\gamma}(t)}(\dot\gamma(t),v)=0\}$.
Therefore each $X\in\gamma^\ast TM$ has a decomposition
$\textsf{tan}_\gamma(X)+ \textsf{nor}_\gamma(X)$, where
\begin{eqnarray}\label{e:TangNormal}
\textsf{tan}_\gamma(X):=\frac{g_{\dot\gamma}(X,\dot\gamma)}{g_{\dot\gamma}(\dot\gamma,
\dot\gamma)}\dot\gamma\quad\hbox{and}\quad
\textsf{nor}_\gamma(X):=X-\frac{g_{\dot\gamma}(X,\dot\gamma)}{g_{\dot\gamma}(\dot\gamma,
\dot\gamma)}\dot\gamma.
\end{eqnarray}
Clearly, $W^{1,2}_{P\times Q}(\gamma^\ast TM)$ is invariant for operations $\textsf{tan}_{\gamma}$ and $\textsf{nor}_\gamma$.
Let
 \begin{eqnarray*}
W^{1,2}_{P\times Q}(\gamma^\ast TM)^\bot:=\{\textsf{nor}_\gamma(X)\,|\,X\in W^{1,2}_{P\times Q}(\gamma^\ast TM)\}.
\end{eqnarray*}
Denote by ${\bf I}^{\gamma,\bot}_{P,Q}$ be the restriction of ${\bf I}^\gamma_{P,Q}$
 to $W^{1,2}_{P\times Q}(\gamma^\ast TM)^\bot$. We have\\

 \noindent{\bf Proposition~B}. {\it
  Let $\gamma:[0,\tau]\to M$ be an $L$-geodesic satisfying (\ref{e:1.4}). Suppose
 $L\circ\dot\gamma\ne 0$. Then
 \begin{description}
\item[(i)] ${\bf I}^\gamma_{P,Q}(\textsf{nor}_\gamma(V), \textsf{tan}_\gamma(W))=0$ for all $V, W\in W^{1,2}_{P\times Q}(\gamma^\ast TM)$.\\
\item[(ii)] ${\rm Ker}({\bf I}^{\gamma,\bot}_{P,Q})={\rm Ker}({\bf I}^\gamma_{P,Q})$,
and ${\rm Index}({\bf I}^{\gamma,\bot}_{P,Q})={\rm Index}({\bf I}^\gamma_{P,Q})$
if $g_{\dot\gamma(t)}$ is positive definite  for some (and so all)  $t$ in $[0,\tau]$,
where ${\rm Index}({\bf I}^{\gamma,\bot}_{P,Q})$ and ${\rm Index}({\bf I}^\gamma_{P,Q})$ are Morse indexes of
${\bf I}^{\gamma,\bot}_{P,Q}$ and ${\bf I}^\gamma_{P,Q}$, respectively.
\end{description}}
Its proof will be given at the end of Section~\ref{sec:2}.

Let $\gamma:[0,\tau]\to M$ be an $L$-geodesic. (It is $C^4$.)
By \cite[Definition~3.12]{Jav15}, a $C^2$
vector field $J$ along $\gamma$ is said to be a  \textsf{Jacobi field} if it satisfies the so-called  \textsf{Jacobi equation}
\begin{equation}\label{e:JacobiEq}
D_{\dot{\gamma}}^{\dot{\gamma}}D_{\dot{\gamma}}^{\dot{\gamma}}J-R_{\dot\gamma}(\dot\gamma,J)\dot\gamma=0.
\end{equation}
(Jacobi fields along $\gamma$ must be $C^4$ because $\gamma$ is $C^4$ and  $R^i_{jkl}$ are $C^{2}$.)
The set $\mathscr{J}_\gamma$ of all Jacobi fields along $\gamma$  is a $2n$-dimensional vector space (\cite[Lemma~3.14]{Jav15}).
For $0\le t_1<t_2\le\tau$ if there exists a nonzero Jacobi field $J$ along $\gamma|_{[t_1,t_2]}$ such that $J$ vanishes at $\gamma(t_1)$ and $\gamma(t_2)$, then $\gamma(t_1)$ and $\gamma(t_2)$ are said to be mutually \textsf{conjugate} along $\gamma|_{[t_1,t_2]}$.

Let $P\subset M$ be as above, and
let $\gamma:[0,\tau]\to M$ be an $L$-geodesic such that
$\gamma(0)\in P$, $\dot\gamma(0)$ is $g_{\dot\gamma(0)}$-orthogonal to $P$,
and  that  $g_{\dot\gamma(0)}|_{T_{\gamma(0)}P\times T_{\gamma(0)}P}$ is nondegenerate.
 A Jacobi field $J$
along $\gamma$ is called  \textsf{$P$-Jacobi} if
\begin{equation}\label{e:P-JacobiField}
J(0)\in T_{\gamma(0)}P\quad\hbox{and}\quad
{\rm tan}^P_{\dot\gamma(0)}\big((D_{\dot{\gamma}}^{\dot{\gamma}}J)(0)\big)=\tilde{S}^P_{\dot\gamma(0)}(J(0)).
\end{equation}
An instant $t_0\in (0,\tau]$ is  called \textsf{$P$-focal} if there exists a non-null $P$-Jacobi field $J$ such that $J(t_0)=0$;
and $\gamma(t_0)$ is said to be a \textsf{$P$-focal point} along $\gamma$.
The dimension of the space $\mathscr{J}^P_\gamma$ of all $P$-Jacobi fields along $\gamma$  is equal to $n=\dim M$.
The dimension $\mu^P(t_0)$ of
$$
\mathscr{J}^P_\gamma(t_0):=\{J\in \mathscr{J}^P_\gamma\,|\, J(t_0)=0\}
$$
is called the (geometrical)  \textsf{multiplicity} of $\gamma(t_0)$.
For convenience we understand $\mu^{P}(t_{0})=0$ if $\gamma(t_0)$ is not a $P$-focal point along $\gamma$.
Then the claim below (\ref{e:secondDiff+})  implies that for any $t\in (0, \tau]$,
\begin{equation}\label{e:P-JacobiField1-}
{\rm Ker}({\bf I}^{\gamma_t}_{P,\gamma(t)})=
\mathscr{J}^P_{\gamma_t}(t)\quad\hbox{with $\gamma_t=\gamma|_{[0,t]}$}.
\end{equation}
In particular, if $g_{\dot\gamma(t)}$ is positive definite for some $t\in [0,\tau]$,
this and Proposition~B(ii) give
\begin{equation}\label{e:P-JacobiField1}
{\rm Ker}({\bf I}^{\gamma,\bot}_{P,q})={\rm Ker}({\bf I}^\gamma_{P,q})=
\mathscr{J}^P_\gamma(\tau)
\quad\hbox{with $q=\gamma(\tau)$}.
\end{equation}

\noindent{\bf 1.3. Main results}.
Recall that  the (Morse) index of a bilinear and symmetric form $I$ on a vector space $E$ is the maximum number of linearly
independent elements of $E$ on which the function $I$ is negative.
Here is our main result.

\begin{theorem}[Morse index theorem]\label{th:MorseIndex}
Let $P, Q$ be two $C^7$ submanifolds in $M$ of dimension less than $n=\dim M$,
and let $\gamma:[0,\tau]\to M$ be an $L$-geodesic which is perpendicular to $P$ at $\gamma(0)$, and $q=\gamma(\tau)$.
Suppose that $g_{\dot\gamma(0)}$ is positive definite (and so each $g_{\dot\gamma(t)}$ is positive definite
and $L\circ\dot\gamma$ is a positive constant by Remark~A below (\ref{e:compat})).
Then
\begin{description}
\item[(i)] $(0,\tau]$ contains only a finite number of instants $t$ such that $\gamma(t)$ are $P$-focal points along $\gamma|_{[0,t]}$.
\item[(ii)] There is a partition $0=t_0<\cdots<t_m=\tau$ such that
$\gamma((0, t_1])$ contains  no $P$-focal points and that
 for any two instants $a<b$ in each interval $[t_i, t_{i+1}]$, $i=0,\cdots, m-1$,
points $\gamma(a)$ and $\gamma(b)$ are not mutually conjugate along $\gamma|_{[a, b]}$.
\item[(iii)]
\begin{equation}\label{e:MS1}
{\rm Index}({\bf I}^\gamma_{P,q})=\sum\limits_{t_{0}\in(0,\tau)}\mu^{P}(t_{0}).
\end{equation}
Moreover,  if $\gamma$ is also perpendicular to  $Q$ at $q=\gamma(\tau)$
and $\{X(\tau)\,|\,X\in\mathscr{J}^P_\gamma\}\supseteq T_{\dot\gamma(\tau)}Q$
(the latter may be satisfied if $\gamma(\tau)$ is not a $P$-focal point),
then
\begin{eqnarray}\label{e:MS2}
&&{\rm Index}({\bf I}^\gamma_{P,Q})={\rm Index}({\bf I}^\gamma_{P,q})+ {\rm Index}(\mathcal{A}_\gamma),\\
&&{\rm Index}({\bf I}^{\gamma,\perp}_{P,q})=\sum\limits_{t_{0}\in(0,\tau)}\mu^{P}(t_{0}),\label{e:MS1bot}\\
&&{\rm Index}({\bf I}^{\gamma,\perp}_{P,Q})={\rm Index}({\bf I}^{\gamma,\perp}_{P,q})+ {\rm Index}(\mathcal{A}_\gamma),\label{e:MS3bot}
\end{eqnarray}
where $\mathcal{A}_\gamma$ is the bilinear symmetric form on $\mathscr{J}^{P}_\gamma$ defined by
$$
\mathcal{A}_\gamma(J_1,J_2)=g_{\dot\gamma(\tau)}(D_{\dot\gamma}^{\dot\gamma}J_1(\tau)+\tilde{S}^Q_{\dot\gamma(\tau)}(J_1(\tau)), J_2(\tau)).
$$
\end{description}
\end{theorem}

\begin{remark}\label{rm:MorseIndex}
{\rm
For a smooth Finsler manifold $(M, F)$, Peter \cite{Pe06} used the Cartan connection to define the Morse index form,
$P$-Jacobi field and the shape operator  in the horizontal subspaces of $T(TM\setminus\{0\})$ and proved (\ref{e:MS1bot})-(\ref{e:MS3bot}).
In this case it holds naturally that $g_{\dot\gamma(t)}$ is positive definite for all $t$; see Corollary~\ref{cor:MorseIndex1}.
Different from our method outlined below (\ref{e:MS5}), Peter's proof is to follow the line of Milnor \cite{Mi} and Piccione and
Tausk \cite{PiTa}. We won't discuss whether Peter's method  is still effective
under our assumption of lower smoothness about $(M, F)$ though parallel vector fields and geodesics are
the same for the Cartan and Chern connections.
}
\end{remark}

When $L$ is equal to the square of a conic Finsler metric on $M$,
since $L$-geodesics must be nonconstant curves, it naturally holds that
$L\circ\dot\gamma>0$ and $g_{\dot\gamma(t)}$ is positive definite  for each $t\in [0,\tau]$.
Hence we have:

\begin{corollary}\label{cor:MorseIndex1}
Let $M$ and $P, Q$ be as in Theorem~\ref{th:MorseIndex}, and let $F$ be a $C^6$ conic Finsler metric on $M$ with domain $A$.
For a nonconstant $F$-geodesic $\gamma:[0,\tau]\to M$ which is perpendicular to $P$ at $\gamma(0)$, and $q=\gamma(\tau)$,
that is, it is a critical point of the functional
$\mathcal{E}_{P,q}:C^1_A([0,\tau];M, P, q)\to\R$ defined by (\ref{e:Fenergy}) with $L=F^2$,
suppose that ${\bf I}^\gamma_{P,q}$ and ${\bf I}^\gamma_{P,Q}$ with $L=F^2$ are defined as above. Then
the conclusions (i)-(ii) and (\ref{e:MS1}) in Theorem~\ref{th:MorseIndex} hold true;
moreover,  if $\gamma$ is also perpendicular to  $Q$ at $q=\gamma(\tau)$
and $\{X(\tau)\,|\,X\in\mathscr{J}^P_\gamma\}\supseteq T_{\dot\gamma(\tau)}Q$
(the latter may be satisfied if $\gamma(\tau)$ is not a $P$-focal point), then
(\ref{e:MS2}), (\ref{e:MS1bot}) and (\ref{e:MS3bot}) also hold true.
\end{corollary}

After proving Theorem~\ref{th:MorseIndex} in next section we shall give some
related conclusions, consequences and examples in Section~\ref{sec:3}.

\section{Proof of Theorem~\ref{th:MorseIndex}}\label{sec:2}

Denote by  $PC^2(\gamma^\ast TM)$  the space of continuous and piecewise $C^2$ sections of the bundle $\gamma^\ast TM$ and by
 $$
 PC^{2}_{P\times Q}(\gamma^\ast TM)=\{
\xi\in PC^{2}(\gamma^\ast TM)\,|\,\xi(0)\in T_{\gamma(0)}P,\;\xi(\tau)\in T_{\gamma(\tau)}Q\}.
$$
Let $\tilde{\bf I}^\gamma_{P,Q}$ be the restriction of ${\bf I}^\gamma_{P,Q}$ to
$PC^{2}_{P\times Q}(\gamma^\ast TM)$. Since $PC^2(\gamma^\ast TM)$ is dense in
 $W^{1,2}_{P\times Q}(\gamma^\ast TM)$ and  ${\rm Ker}({\bf I}^\gamma_{P,Q})\subset PC^2(\gamma^\ast TM)$ we get
\begin{equation}\label{e:MorseIndexNullity1}
{\rm Ker}(\tilde{\bf I}^{\gamma}_{P,Q})={\rm Ker}({\bf I}^\gamma_{P,Q})\quad\hbox{and}\quad
{\rm Index}(\tilde{\bf I}^{\gamma}_{P,Q})={\rm Index}({\bf I}^\gamma_{P,Q}).
\end{equation}
Therefore (\ref{e:MS1}) is equivalent to the equality
\begin{equation}\label{e:MS5}
{\rm Index}(\tilde{\bf I}^\gamma_{P,q})=\sum\limits_{t_{0}\in(0,\tau)}\mu^{P}(t_{0}).
\end{equation}

Our proof will be completed in four steps.
Firstly, we shall use the Morse index theorem by Hermann (Theorem~\ref{th:Hermann}) and a result by
 Hartman and Wintner \cite[page 77]{HaWin} to prove the conclusions (i)-(ii) and
Proposition~\ref{prop:MorseIndex}. The latter shows that the index of the restriction of $\tilde{\bf I}^\gamma_{P,q}$ to
a subspace $PC^2_{P\times q}(\gamma^\ast TM)_0$ of $PC^2_{P\times q}(\gamma^\ast TM)$ is equal to
the right side of (\ref{e:MS5}). Next, we prove that (\ref{e:MS5}) easily follows from Proposition~\ref{prop:MorseIndex}
and Proposition~\ref{prop:Sak3.1}. The latter claims that $\tilde{\bf I}^\gamma_{P,q}$ and the restriction of $\tilde{\bf I}^\gamma_{P,q}$ to a finite dimensional subspace $PC^2_{P\times q}(\gamma^\ast TM)_{00}$ of $PC^2_{P\times q}(\gamma^\ast TM)_0$ have the same nullity and index.
In Step~3 we shall prove Proposition~\ref{prop:Sak3.1} through several lemmas which are closely related to
\cite{Jav15}. Finally, we use the same ideas as those of \cite[Theorem~1.2]{Pe06} to complete the proofs of
 the second part in (iii). The so-called index lemma (\cite[Lemma~4.2]{Pe06}) in Finsler geometry
 played a key role in the proof of \cite[Theorem~1.1]{Pe06}. Because of using Theorem~\ref{th:Hermann}
  we avoid to use the pseudo-Finsler version (Proposition~\ref{prop:Sa2.10}) of \cite[Lemma~4.2]{Pe06}.\\

\noindent{\bf Step 1}(\textsf{Prove the conclusions (i)-(ii) and the following proposition}).

\begin{proposition}\label{prop:MorseIndex}
Let $\gamma:[0,\tau]\to M$ be an $L$-geodesic which is perpendicular to $P$ at $\gamma(0)$, $q=\gamma(\tau)$
and let $g_{\dot\gamma(0)}$ (and so each $g_{\dot\gamma(t)}$) be  positive definite.
Then the index of the restriction of $\tilde{\bf I}^\gamma_{P,q}$ to
\begin{equation}\label{e:subspace}
PC^2_{P\times q}(\gamma^\ast TM)_0:=\{V\in PC^2_{P\times q}(\gamma^\ast TM)\,|\,
{\rm tan}^P_{\dot\gamma}(D_{\dot{\gamma}}^{\dot{\gamma}}V)(0)=\tilde{S}^P_{\dot\gamma(0)}(V(0))\}.
\end{equation}
 is equal to
\begin{equation*}
\sum\limits_{t_{0}\in(0,\tau)}\mu^{P}(t_{0}).
\end{equation*}
\end{proposition}
If $P$ is a point $p=\gamma(0)$, then
$PC^2_{p\times q}(\gamma^\ast TM)_0=PC^2_{p\times q}(\gamma^\ast TM)$
and Proposition~\ref{prop:MorseIndex} gives  (\ref{e:MS5}) and hence
(\ref{e:MS1}).

Proposition~\ref{prop:MorseIndex} will be derived from
 the following Morse index theorem proved in \cite[\S31]{Her68} by R.A. Hermann.

\begin{theorem}[Morse index theorem]\label{th:Hermann}
Let $E$ be a finite-dimensional real vector space with inner product $\langle\cdot,\cdot\rangle$,
$E_0\subset E$ a subspace, $\mathcal{Q}:E_0\times E_0\to\mathbb{R}$ a bilinear, symmetric form.\footnote{
In \cite[\S31]{Her68}, $E$, $E_0$, $\mathcal{Q}$ and $\tau$
were denoted by $V$, $W$, $Q$ and $a$, respectively. }
Denote by $\Omega([0,\tau], E)$  the space of continuous,
piecewise $C^2$ curves $v:[0,\tau]\to E$  satisfying the following
conditions:
$$
v(0)\in E_0,\quad\langle \dot{v}(0),w\rangle= -\mathcal{Q}(v(0), w)\;\forall w\in E_0,\quad v(\tau)=0.
$$
It is a real vector space.
Given a continuous\footnote{It was not required in \cite[\S31]{Her68} that $t\mapsto R_t$ is continuous.
 However,  this condition is necessary in order to ensure that
the integrations of $\langle R_t(v(t)), v(t)\rangle$ appeared on \cite[page~403]{Her68} and other places
are well-defined.} map $[0,\tau]\ni t\mapsto {R}_t\in\mathscr{L}_s(E)$, where
$\mathscr{L}_s(E)$ consists of  linear transformations
$L:E\to E$ satisfying $\langle L(u), v\rangle=\langle u, L(v)\rangle$ for any $u,v\in E$,
consider the quadratic form
$$
\mathcal{I}(v)=-\mathcal{Q}(v(0),v(0))+\int^\tau_0\left[\langle \dot{v}(t),\dot{v}(t)\rangle-\langle{R}_t(v(t)), v(t)\rangle\right] dt
$$
for $v\in\Omega([0,\tau],E)$, and the differential operator
$$
\mathcal{J}=\frac{d^2}{dt^2}+{R}_t.
$$
 Define the \textsf{Morse index} of $\mathcal{I}$ as the maximum number of linearly
independent elements of $\Omega([0,\tau],E)$ on which the function $\mathcal{I}$ is negative.
A point $a\in(0,\tau]$ is said to be a \textsf{{focal point}} for the operator $\mathcal{J}$
and boundary condition $(E_0, \mathcal{Q})$ if there is a nontrivial $C^2$
curve $u:[0,a]\to E$  satisfying
\begin{eqnarray*}
\mathcal{J}(u) = 0,\quad u(0)\in E_0,\quad
\langle \dot{u}(0), w\rangle = -\mathcal{Q}(u(0), w)\;\forall w\in E_0,\quad u(a) = 0.
\end{eqnarray*}
(Clearly, $u$ can uniquely extend a $C^2$ curve $\tilde{u}:[0,\tau]\to E$ satisfying $\mathcal{J}(\tilde{u})=0$
by the continuation theorem for solutions of ordinary differential equations.)
The \textsf{multiplicity} of such a focal point is equal to the dimension of the linear space of
all curves satisfying these conditions (hence, it is always finite and no greater
than the dimension of $E$). Then
the index of $\mathcal{I}$ is finite and equal to the sum of multiplicities of the
focal points contained in the open interval $(0,\tau)$. It is also equal to the maximal
number of linearly independent elements of $\Omega([0,\tau], E)$ that are $C^2$ and
are eigenfunctions of the differential operator $\mathcal{J}$ for positive eigenvalues.
\end{theorem}

\begin{corollary}[\hbox{\cite[Corollary of Lemma~31.8]{Her68}}]\label{cor:Hermann}
Under the assumptions of Theorem~\ref{th:Hermann} the interval $[0,\tau]$
contains no focal points if and only if $\mathcal{I}(u)>0$ for all curves $u\in\Omega([0,\tau], E)\setminus\{0\}$.
\end{corollary}

\begin{lemma}[\hbox{\cite[Lemmas~31.5 and 31.6]{Her68}}]\label{lem:31.5-6}
Let the assumptions of Theorem~\ref{th:Hermann} be satisfied. If $\varepsilon\in (0,\tau]$ is sufficiently small,
there are no focal points on the interval $[0,\varepsilon]$ for $\mathcal{J}$
and  $(E_0, \mathcal{Q})$. Moreover,
if $t_0\in (0,\tau]$ is a focal point and $\varepsilon>0$ is small enough, then
the interval $[t_0-\varepsilon, t_0+\varepsilon]\cap[0,\tau]$ contains no other focal point.
Consequently, $[0,\tau]$ contains only a finite number of focal points for $\mathcal{J}$
and  $(E_0, \mathcal{Q})$.
 \end{lemma}

\begin{definition}[\hbox{\cite[page~411, Definition]{Her68}}]\label{def:Hermann}
{\rm (Let $R_t$ be as in Theorem~\ref{th:Hermann}.) Reals $a<b$ in $[0,\tau]$ are said to be mutually conjugate if there is a $C^2$
curve $v(t)$, not identically zero, satisfying $\ddot{v}(t)+ R_t(v(t))=0\;\forall t\in [a,b]$
and $v(a)=v(b)=0$.}
\end{definition}

Clearly, a point $a\in(0,\tau]$ is a focal point for the operator $\mathcal{J}$
and boundary condition $(E_0, \mathcal{Q})=(\{0\}, 0)$ if and only if $0$ and $a$  are mutually conjugate.
If $t_0\in (0,\tau]$ and $0$ are mutually conjugate, then Lemma~\ref{lem:31.5-6} implies that
for a sufficiently small $\varepsilon>0$, $0$ and each point $t\ne t_0$ in the interval $[t_0-\varepsilon, t_0+\varepsilon]\cap[0,\tau]$
 are not mutually conjugate. For an interval $I\subset [0,\tau]$,
 if no nontrivial solution $u\in C^2(I,E)$ of $\mathcal{J}u=0$ vanishes twice in $I$, we say $\mathcal{J}$ to be
\textsf{disconjugate} on $I$.
 Let $0\le a<b\le\tau$. According to (IV) in \cite[page 77]{HaWin} by Hartman and Wintner,
if there exists a continuous differentiable map $[a,b]\ni t\mapsto V(t)\in\mathscr{L}_s(E)$
 such that $-V'(t) + \frac{1}{4}(V(t))^2 + R_t\le 0$ for all $t\in [a,b]$,
  (hereafter for $P\in\mathscr{L}_s(E)$ we write $P\le 0$ if $\langle Pu,u\rangle\le 0$ for all $u\in E$), then
$\mathcal{J}$ is disconjugate on $[a,b]$. Given $t_0\in [0,\tau]$ let $V(t)=(t-t_0)C{\rm id}_E$, where $C>0$.
Then
$$
-V'(t) + \frac{1}{4}(V(t))^2 + R_t=-C{\rm id}_E+ \frac{1}{4}(t-t_0)^2C^2{\rm id}_E+ R_t.
$$
 Since $t\mapsto R_t$ is continuous, we may choose a large $C>0$ and a small $\varepsilon>0$ such that
 $$
 -C{\rm id}_E+ \frac{1}{4}(t-t_0)^2C^2{\rm id}_E+ R_t\le 0,\quad\forall t\in I:=[0,\tau]\cap[t_0-\varepsilon,t_0+\varepsilon].
 $$
  It follows that each point in $[0,\tau]$ has a neighborhood such that $\mathcal{J}$ is disconjugate on this neighborhood.
 This and Lemma~\ref{lem:31.5-6} imply

 \begin{lemma}\label{lem:new-add}
Under the assumptions of Theorem~\ref{th:Hermann}, there is a partition $0=t_0<\cdots<t_m=\tau$ such that
there are no focal points on the interval $[0, t_1]$ for $\mathcal{J}$
and  $(E_0, \mathcal{Q})$, and that $\mathcal{J}$ is disconjugate on each interval $[t_i, t_{i+1}]$, $i=0,1,\cdots, m-1$.
 \end{lemma}

\begin{lemma}[\hbox{\cite[Lemma~31.9]{Her68}}]\label{lem:31.9}
Under the assumptions of Theorem~\ref{th:Hermann}, let reals $a<b$ in $[0,\tau]$ be
such that $[a,b]$ contains no pair of mutually conjugate points.
Suppose that $u(t)$ and $v(t)$, $a\le t\le b$, are continuous curves
such that $u$ is piecewise $C^2$ and $v$ is $C^2$. Moreover, $v$ satisfies
$$
\ddot{v}(t)+ R_t(v(t))=0\;\forall t\in [a,b],\quad u(a)=v(a),\quad u(b)=v(b).
$$
Then
$$
\int^b_a\left[\langle \dot{v}(t),\dot{v}(t)\rangle-\langle{R}_t(v(t)), v(t)\rangle\right] dt\le
\int^b_a\left[\langle \dot{u}(t),\dot{u}(t)\rangle-\langle{R}_t(u(t)), u(t)\rangle\right] dt.
$$
Equality holds only if $u(t)=v(t)$ for all $t\in [a,b]$.
 \end{lemma}

Since $\tilde{\bf I}^\gamma_{P,q}$ is  the restriction of ${\bf I}^\gamma_{P,q}$ to
$PC^{2}_{P\times q}(\gamma^\ast TM)$, by (\ref{e:secondDiff})
\begin{eqnarray}\label{e:secondDifff2}
\tilde{\bf I}^\gamma_{P,q}(V, W)&=&
\int_0^\tau \left(g_{\dot\gamma}(R_{\dot\gamma}(\dot\gamma,V)\dot\gamma, W)+g_{\dot\gamma}(D_{\dot\gamma}^{\dot\gamma}V,D_{\dot\gamma}^{\dot\gamma}W)\right)dt
\nonumber\\
&&+g_{\dot\gamma(0)}(\tilde{S}^P_{\dot\gamma(0)}(V(0)),W(0)).
 \end{eqnarray}
 Note that  $\tilde{\bf I}^\gamma_{P,q}$ has still the kernel as in (\ref{e:P-JacobiField1}),
that is, ${\rm Ker}(\tilde{\bf I}^\gamma_{P,q})=\mathscr{J}^P_\gamma(\tau)$.

Because of the splitting (\ref{e:splitting}) with $p=\gamma(0)$ and $v=\dot\gamma(0)$,
we may use the positive definiteness of $g_{\dot\gamma(0)}$ to construct a normalized  orthonormal basis
for $(T_{\gamma(0)}M, g_{\dot\gamma(0)})$,
$e_1,\cdots, e_n$, that is, a basis satisfying $g_{\dot\gamma(0)}(e_i,e_j)=\delta_{ij}$,
 where  $\delta_{ij}$ is the Kronecker's delta for $i,j=1,\cdots,n$.
Clearly, we can also require  that $e_1,\cdots,e_k$ is  a normalized orthonormal basis for
 $(T_{\gamma(0)}P, g_{\dot\gamma(0)}|_{T_{\gamma(0)}P\times T_{\gamma(0)}P})$
and that $e_{k+1},\cdots,e_n$ be  a normalized orthonormal basis for $(T_{\gamma(0)}P)^\perp_{\dot\gamma(0)}$ and $g^\bot_{\dot\gamma(0)}$.
Since $\gamma$ is $C^4$ and $\Gamma^m_{ij}$ is $C^{3}$,
as in \cite[page~74]{Jav15} we may construct
 a $C^4$ parallel orthonormal frame  fields $E_1,\cdots,E_n$ along $\gamma$
starting at $e_1,\cdots,e_n$. (Therefore $g_{\dot\gamma(t)}(E_i(t),E_j(t))=\delta_{ij}$
because $t\mapsto g_{\dot\gamma(t)}(E_i(t),E_j(t))$ is constant.)
 Write $V=\sum_iv^iE_i$ and $W=\sum_iw^iE_i$, and denote by
 $\langle\cdot,\cdot\rangle_{\mathbb{R}^n}$ the standard Euclidean metric.
Then
$$
D_{\dot\gamma}^{\dot\gamma}V(t)=\sum_i({v}^i)'(t)E_i,\quad
 D_{\dot\gamma}^{\dot\gamma}W(t)=\sum_i({w}^i)'(t)E_i
 $$
 and therefore
 (\ref{e:secondDifff2}) becomes into
\begin{eqnarray} \label{e:secondDifff3}
\tilde{\bf I}^\gamma_{P,q}(V, W)&=&\int^\tau_0
\langle\dot{v},\dot{w}\rangle_{\R^n}dt-\int_0^\tau \langle\mathfrak{R}_tv(t), w(t)\rangle_{\R^n}dt
\nonumber\\
&&+g_{\dot\gamma(0)}(\tilde{S}^P_{\dot\gamma(0)}(V(0)),W(0)),
 \end{eqnarray}
 where $v(t)=(v^1(t),\cdots,v^n(t))$, $w(t)=(w^1(t),\cdots,w^n(t))$ and  $\mathfrak{R}_t:\R^n\to\R^n$ is given by
 $$
 \mathfrak{R}_tx=\left(\sum_i\mathfrak{R}_t^{i1}x^i,\cdots,
    \sum_i\mathfrak{R}_t^{in}x^i\right)
 $$
 with $\mathfrak{R}_t^{ij}=-g_{\dot\gamma}(R_{\dot\gamma}(\dot\gamma,E_i)\dot\gamma, E_j)(t)$.
Clearly, $t\mapsto \mathfrak{R}_t^{ij}$ is $C^2$ because $R^i_{jkl}$, $\dot\gamma(t)$ and $E_i(t)$ are $C^{2}$, $C^3$ and $C^4$,
respectively.
By (\ref{e:Rsym7}), the matrix $(\mathfrak{R}_t^{ij})$ is also symmetric, i.e.,
\begin{equation*}
g_{\dot\gamma(t)}(R_{\dot\gamma(t)}(\dot\gamma(t),E_i(t))\dot\gamma(t), E_j(t))=
g_{\dot\gamma(t)}(R_{\dot\gamma(t)}(\dot\gamma(t),E_j(t))\dot\gamma(t), E_i(t)),
\quad\forall i,j.
\end{equation*}
Moreover, since $W(0), V(0)$ and $\tilde{S}^P_{\dot\gamma(0)}(V(0))$ belong to $T_{\gamma(0)}P$, we deduce that
$v(0)$ and $w(0)$  belong to the subspace
 of $\R^n$,
 $\R^k\equiv\{(x^1,\cdots,x^n)\in\R^n\,|\, x^{k+1}=\cdots x^n=0\}$,
 and
 \begin{eqnarray*}
 \mathfrak{Q}({v}(0), {w}(0)):
 =-\sum^k_{i=1}\sum^k_{j=1}g_{\dot\gamma(0)}\left(\tilde{S}^P_{\dot\gamma(0)}(e_i),e_j\right)v^i(0)w^j(0)
 =-g_{\dot\gamma(0)}(\tilde{S}^P_{\dot\gamma(0)}(V(0)),W(0)).
   \end{eqnarray*}
   Clearly, $\mathfrak{Q}$ is a symmetric bilinear form on $\R^k$ and
  (\ref{e:secondDifff3}) becomes into
\begin{eqnarray}\label{e:secondDifff4}
\tilde{\bf I}^\gamma_{P,q}(V, W)=-\mathfrak{Q}({v}(0), {w}(0))+\int^\tau_0
\langle\dot{v},\dot{w}\rangle_{\R^n}dt-\int_0^\tau \langle\mathfrak{R}_tv(t), w(t)\rangle_{\R^n}dt.
 \end{eqnarray}
 For a $C^2$ vector field $J$ along $\gamma$, we can write $J(t)=\sum^n_{i=1}v^i(t)E_i(t)$.
 Then $J$ satisfies the Jacobi equation (\ref{e:JacobiEq}) if and only if
 $v(t)=(v^1(t),\cdots,v^n(t))$ satisfies
\begin{equation}\label{e:JacobiEq1}
\ddot{v}(t)+ \mathfrak{R}_tv(t)=0.
\end{equation}
 When $J$ is also $P$-Jacobi, the boundary condition (\ref{e:P-JacobiField}) becomes
\begin{equation}\label{e:P-JacobiField2}
v(0)\in \mathbb{R}^k\quad\hbox{and}\quad
\quad\langle \dot{v}(0),w\rangle= -\mathfrak{Q}({v}(0), {w})\;\forall w\in \mathbb{R}^k.
\end{equation}
Thus $J$ is a $P$-Jacobi field along $\gamma$ if and only if $v$ satisfies (\ref{e:JacobiEq1}) and (\ref{e:P-JacobiField2}).

In Theorem~\ref{th:Hermann} let us take
\begin{center}
$E=\mathbb{R}^n$, $\quad\langle\cdot,\cdot\rangle=\langle\cdot,\cdot\rangle_{\mathbb{R}^n}$,
$\quad E_0=\mathbb{R}^k$, $\quad\mathcal{Q}=\mathfrak{Q}\quad$ and $\quad\mathcal{J}=\frac{d^2}{dt^2}+\mathfrak{R}_t$
\end{center}
and therefore the corresponding index form
\begin{eqnarray}\label{e:two-forms0}
\mathcal{I}(v)=-\mathfrak{Q}(v(0),v(0))+\int^\tau_0\left[\langle \dot{v}(t),\dot{v}(t)\rangle_{\R^n}-\langle{\mathfrak{R}}_t(v(t)), v(t)\rangle_{\R^n}\right] dt.
\end{eqnarray}
Then (\ref{e:secondDifff4}) implies that
\begin{eqnarray}\label{e:two-forms}
\mathcal{I}(v)=\tilde{\bf I}^\gamma_{P,q}\left(\sum^n_{i=1}v^iE_i, \sum^n_{i=1}v^iE_i\right),\quad\forall v\in\Omega([0,\tau],\R^n).
\end{eqnarray}
The arguments in last paragraph show:

\begin{proposition}\label{prop:two-Jacobi}
 $a\in(0,\tau]$ is  a focal point for the operator $\mathcal{J}=\frac{d^2}{dt^2}+\mathfrak{R}_t$
and boundary condition $(E_0, \mathcal{Q})=(\R^k, \mathfrak{Q})$
if and only if $\gamma(a)$ is a $P$-focal point along $\gamma$; and the multiplicity of $a$
as a focal point for $\mathcal{J}$ and $(\R^k, \mathfrak{Q})$ is equal to that of $\gamma(a)$ as a $P$-focal point, $\mu^P(a)$.
Moreover
$$
\Omega([0,\tau],\R^n)\ni (v^1,\cdots,v^n)\to \sum^n_{i=1}v^iE_i\in PC^2(\gamma^\ast TM)_0
$$
is a linear isomorphism.
\end{proposition}

\begin{proof}[\bf Proofs of the conclusions (i)-(ii) and Proposition~\ref{prop:MorseIndex}]
Because of Proposition~\ref{prop:two-Jacobi},  the conclusions (i) and (ii) follow from Lemma~\ref{lem:31.5-6} and \ref{lem:new-add}, respectively.

By Theorem~\ref{th:Hermann} the index of the form $\mathcal{I}$ in (\ref{e:two-forms0}) is finite and equal to the sum of multiplicities of the
focal points contained in the open interval $(0,\tau)$. These and the first two claims in Proposition~\ref{prop:two-Jacobi} lead to
$$
{\rm Index}(\mathcal{I})=\sum\limits_{t_{0}\in(0,\tau)}\mu^{P}(t_{0}).
$$
From (\ref{e:two-forms}), the definition of index and the third claim in Proposition~\ref{prop:two-Jacobi}
it follows that ${\rm Index}(\mathcal{I})$ is equal to the index of
the restriction of $\tilde{\bf I}^\gamma_{P,q}$ to
$PC^2_{P\times q}(\gamma^\ast TM)_0$.
 The desired conclusion in Proposition~\ref{prop:MorseIndex} is obtained.
({\it Note}: Proposition~B in \S1.2 is not used above.)
\end{proof}

\noindent{\bf Step 2}.
Let $0=t_0<\cdots<t_m=\tau$ be a partition as in (ii).
 By Corollary~\ref{cor:Hermann} we have
\begin{eqnarray}\label{e:positive-focalA}
\tilde{\bf I}^{\gamma_0}_{P,\gamma(t_1)}(V, V)&=&
\int_0^{t_1}\left(g_{\dot\gamma}(R_{\dot\gamma}(\dot\gamma,V)\dot\gamma, V)+g_{\dot\gamma}(D_{\dot\gamma}^{\dot\gamma}V,D_{\dot\gamma}^{\dot\gamma}V)\right)dt
\nonumber\\
&&+g_{\dot\gamma(0)}(\tilde{S}^P_{\dot\gamma(0)}(V(0)), V(0))>0,\\
&&\quad\forall V\in PC^2_{P\times \gamma(t_1)}(\gamma_0^\ast TM)_0\setminus\{0\},\label{e:positive-focalB}
\end{eqnarray}
where $PC^2_{P\times \gamma(t_1)}(\gamma_0^\ast TM)_0$ is as in (\ref{e:subspace}).
Similarly, for each $i=1,\cdots,m-1$, Lemma~\ref{lem:31.9} implies
\begin{eqnarray}\label{e:positive-conjA}
&&\tilde{\bf I}^{\gamma_i}_{\gamma(t_i),\gamma(t_{i+1})}(V, V)=
\int_{t_i}^{t_{i+1}}\left(g_{\dot\gamma}(R_{\dot\gamma}(\dot\gamma,V)\dot\gamma, V)+g_{\dot\gamma}(D_{\dot\gamma}^{\dot\gamma}V,D_{\dot\gamma}^{\dot\gamma}V)\right)dt>0,\\
&&\quad\forall V\in PC^2(\gamma_i^\ast TM)\setminus\{0\}\;\hbox{satisfying}\;V(t_i)=0\;\hbox{and}\;V(t_{i+1})=0.
\label{e:positive-conjB}
\end{eqnarray}
Let
$PC^2_{P\times q}(\gamma^\ast TM)_{00}$
consist of $V\in PC^2_{P\times q}(\gamma^\ast TM)_0$ such that
$V|_{[t_0,t_1]}$ is a $P$-Jacobi field along $\gamma_0$ and
that $V|_{[t_i, t_{i+1}]}$ ($1\le i\le m-1$) are Jacobi fields.
It is a vector subspace of $PC^2_{P\times q}(\gamma^\ast TM)_0$ of finite dimension.
We shall prove:

\begin{proposition}\label{prop:Sak3.1}
Under the assumptions of Proposition~\ref{prop:MorseIndex} let $\dim P>0$.
Then the null space of $\tilde{\bf I}^\gamma_{P,q}|_{PC^2_{P\times q}(\gamma^\ast TM)_{00}}$
coincides with the null space of $\tilde{\bf I}^\gamma_{P,q}$, which is given by
$\{Y\in\mathscr{J}^P_\gamma\,|\, Y(\tau)=0\}$. Moreover,
the index of $\tilde{\bf I}^\gamma_{P,q}|_{PC^2_{P\times q}(\gamma^\ast TM)_{00}}$ is equal to
the index of $\tilde{\bf I}^\gamma_{P,q}$, and hence finite.
\end{proposition}

We first admit it and give:

\begin{proof}[\bf Proof of (\ref{e:MS5})]
As remarked below Proposition~\ref{prop:MorseIndex} we may assume $\dim P>0$.
By the definition of the Morse index of a bilinear and symmetric form  on a vector space we directly deduce
$$
{\rm Index}\left(\tilde{\bf I}^\gamma_{P,q}|_{PC^2_{P\times q}(\gamma^\ast TM)_{00}}\right)\le
{\rm Index}\left(\tilde{\bf I}^\gamma_{P,q}|_{PC^2_{P\times q}(\gamma^\ast TM)_{0}}\right)\le
{\rm Index}\left(\tilde{\bf I}^\gamma_{P,q}\right)\le {\rm Index}\left({\bf I}^\gamma_{P,q}\right).
$$
The final inequality is actually an equality since $PC^2_{P\times q}(\gamma^\ast TM)$ is dense in $W^{1,2}_{P\times q}(\gamma^\ast TM)$.
(\ref{e:MS5}) and hence (\ref{e:MS1}) follows from these and Propositions~\ref{prop:MorseIndex} and \ref{prop:Sak3.1} immediately.
\end{proof}

\noindent{\bf Spep 3}(\textsf{Prove Proposition~\ref{prop:Sak3.1}}).
To this end we need some lemmas. By Remark~A, $L(\dot\gamma(t))\equiv L(\dot\gamma(0))$ is a positive constant.
We may define $\textsf{tan}_{\gamma}(X)$ and $\textsf{nor}_\gamma(X)$  as in (\ref{e:TangNormal}).
By the proofs of \cite[Lemmas~3.16,3.17(ii), Prop.3.18]{Jav15} we immediately obtain (i), (ii) and (iii)
of the following lemma, respectively.

\begin{lemma}\label{orthogJacobi}
Let $\gamma:[0,\tau]\to M$ be an  $L$-geodesic (and hence $C^4$) satisfying $L\circ\dot\gamma\ne 0$.
 Then:
\begin{enumerate}
\item[\rm (i)] $D_{\dot{\gamma}}^{\dot{\gamma}}\big(\textsf{tan}_{\gamma}(J)\big)=\textsf{tan}_{\gamma}(D_{\dot{\gamma}}^{\dot{\gamma}}J)$ and
$D_{\dot{\gamma}}^{\dot{\gamma}}\big(\textsf{nor}_\gamma(J)\big)=\textsf{nor}_\gamma(D_{\dot{\gamma}}^{\dot{\gamma}}J)$
for $J\in W^{1,2}(\gamma^\ast TM)$.
\item[\rm (ii)] A $C^2$ vector field $J$ along $\gamma$ is a Jacobi field if and only if $\textsf{nor}_\gamma(J)$ and $\textsf{tan}_{\gamma}(J)$ are Jacobi fields.
\item[\rm (iii)] If $J_1$ and $J_2$ are $C^2$ Jacobi fields along $\gamma$, then
 the function  $g_{\dot\gamma}(J_1,D_{\dot{\gamma}}^{\dot{\gamma}}J_2)-g_{\dot\gamma}(D_{\dot{\gamma}}^{\dot{\gamma}}J_1,J_2)$ is constant.
\end{enumerate}
\end{lemma}

\begin{lemma}\label{lem:BoundaryBase}
Under the assumptions of Lemma~\ref{orthogJacobi},
let $PC^2_{P\times q}(\gamma^\ast TM)_0$ be as in (\ref{e:subspace}).
Then
$$
g_{\dot\gamma}(D_{\dot{\gamma}}^{\dot{\gamma}}V(0), W(0))=g_{\dot\gamma}(V(0),D_{\dot{\gamma}}^{\dot{\gamma}}W(0)),\quad
\forall V,W\in PC^2_{P\times q}(\gamma^\ast TM)_0.
$$
Moreover, if $V, W\in \mathscr{J}^P_\gamma$, then $g_{\dot\gamma}(D_{\dot\gamma}^{\dot\gamma}V(t), W(t))=g_{\dot\gamma}(V(t), D_{\dot\gamma}^{\dot\gamma}W(t))$ for all $t\in [0,\tau]$.
\end{lemma}
\begin{proof}
Since $V(0), W(0)\in T_{\dot\gamma(0)}P$ and $\tilde{S}^P_{\dot\gamma(0)}$ is symmetric we have
\begin{eqnarray*}
g_{\dot\gamma}(D_{\dot{\gamma}}^{\dot{\gamma}}V(0), W(0))&=&
g_{\dot\gamma}({\rm tan}^P_{\dot\gamma}(D_{\dot{\gamma}}^{\dot{\gamma}}V(0)), W(0))\\
&=&g_{\dot\gamma}(\tilde{S}^P_{\dot\gamma(0)}(V(0)), W(0))\\
&=&g_{\dot\gamma}(V(0), \tilde{S}^P_{\dot\gamma(0)}(W(0))\\
&=&g_{\dot\gamma}(V(0), {\rm tan}^P_{\dot\gamma(0)}(D_{\dot{\gamma}}^{\dot{\gamma}}W(0)))
=g_{\dot\gamma}(V(0), D_{\dot{\gamma}}^{\dot{\gamma}}W(0)).
\end{eqnarray*}
The second claim follows from the first one and Lemma~\ref{orthogJacobi}(iii).
\end{proof}

 By Remark~A, $L\circ\dot\gamma(t)$ is a positive constant.  Note that
 $(\dot\gamma^\bot)_0=\{v\in T_{\gamma(0)}M\,|\, g_{\dot{\gamma}(0)}(\dot\gamma(0),v)=0\}\supseteq T_{\gamma(0)}P$.
 We may choose a normalized orthonormal basis for $(\dot\gamma^\bot)_0$ and $g_{\dot\gamma(0)}$,
 $e_1,\cdots,e_{n-1}$, such that $e_1,\cdots,e_k$ is a $g_{\dot\gamma(0)}$-orthonormal basis for $T_{\gamma(0)}P$.
 Therefore $e_1,\cdots,e_{n-1}, e_n=\dot\gamma(0)/\sqrt{L\circ\dot\gamma}$ is
a $g_{\dot\gamma(0)}$-orthonormal basis for $T_{\gamma(0)}M$.

\begin{lemma}\label{lem:base}
Let $\gamma:[0,\tau]\to M$ be an  $L$-geodesic which  starts from $P$  perpendicularly.
Suppose that  $g_{\dot\gamma(0)}$ (and so each $g_{\dot\gamma(t)}$) is positive definite, and that   $e_1,\cdots, e_n$ are as in the above paragraph.
Then there exists a basis of $\mathscr{J}^P_\gamma$,
$J_1,\cdots,J_n$ satisfying the following conditions:
\begin{eqnarray}
&&J_i(0)=e_i\quad\hbox{and}\quad D_{\dot{\gamma}}^{\dot{\gamma}}J_i(0)=\tilde{S}^P_{\dot\gamma(0)}(e_i),\quad
i=1,\cdots,k,\label{e:Jacobibase1}\\
&&J_i(0)=0\quad\hbox{and}\quad D_{\dot{\gamma}}^{\dot{\gamma}}J_i(0)=e_i,\quad
i=k+1,\cdots,n-1,\label{e:Jacobibase2}\\
&&J_n(t)=t\dot\gamma(t)\;\forall t\;\hbox{and so}\;D_{\dot{\gamma}}^{\dot{\gamma}}J_n(0)=\dot\gamma(0)=e_n, \label{e:Jacobibase3}\\
&&\textsf{nor}_\gamma(J_i)=J_i,\; i=1,\cdots,n-1.\label{e:Jacobibase4}
\end{eqnarray}
\end{lemma}
\begin{proof}
 Since for any $v,w\in T_{\gamma(0)}P$ there exists
a unique $C^4$ Jacobi field $J$ along $\gamma$ such that $J(0)=v$ and $D_{\dot{\gamma}}^{\dot{\gamma}}J(0)=w$
(cf. \cite[Lemma~3.14]{Jav15}), we have $C^4$ Jacobi fields  $J_i$ along $\gamma$ satisfying
(\ref{e:Jacobibase1}) and (\ref{e:Jacobibase2}).
Because $\tilde{S}^P_{\dot\gamma(0)}(e_i)\in T_{\gamma(0)}P$ for $i=1,\cdots,k$, and
$e_i$ is $g_{\dot\gamma(0)}$-orthonormal to $T_{\gamma(0)}P$ for $i=k+1,\cdots,n$,
that is, $(T_{\gamma(0)}P)^\perp_{\dot\gamma(0)}={\rm Span}(\{e_{k+1},\cdots,e_n\})$,
from (\ref{e:Jacobibase1})-(\ref{e:Jacobibase2}) we deduce
\begin{eqnarray*}
&&{\rm tan}_{\dot\gamma(0)}\left((D_{\dot{\gamma}}^{\dot{\gamma}}J_i)(0)\right)=D_{\dot{\gamma}}^{\dot{\gamma}}J_i(0)=
\tilde{S}^P_{\dot\gamma(0)}(J_i(0)),\quad
i=1,\cdots,k,\\
&&{\rm tan}_{\dot\gamma(0)}\left((D_{\dot{\gamma}}^{\dot{\gamma}}J_i)(0)\right)=
{\rm tan}_{\dot\gamma(0)}(e_i)=0=\tilde{S}^P_{\dot\gamma(0)}(J_i(0)),\quad
i=k+1,\cdots,n-1.
\end{eqnarray*}
Namely, $J_i$, $i=1,\cdots,n-1$, are $P$-Jacobi fields. By Lemma~\ref{orthogJacobi}(ii),
$\textsf{nor}_\gamma(J_i)$, $i=1,\cdots,n-1$, are Jacobi fields along $\gamma$.
Observe that
\begin{eqnarray*}
&&
\textsf{nor}_\gamma(J_i)(0)=e_i-\frac{g^F_{\dot\gamma}(e_i,\dot\gamma(0))}{g^F_{\dot\gamma}(\dot\gamma(0),
\dot\gamma(0))}\dot\gamma(0)=e_i=J_i(0),\quad i=1,\cdots,k,\\
&&\textsf{nor}_\gamma(J_i)(0)=0-\frac{g^F_{\dot\gamma}(0,\dot\gamma(0))}{g^F_{\dot\gamma}(\dot\gamma(0),
\dot\gamma(0))}\dot\gamma(0)=0=J_i(0),\quad i=k+1,\cdots,n-1,
\end{eqnarray*}
and that  Lemma~\ref{orthogJacobi}(i) implies
\begin{eqnarray*}
(D_{\dot{\gamma}}^{\dot{\gamma}}\textsf{nor}_\gamma(J_i))(0)&=&
\textsf{nor}_\gamma(D_{\dot{\gamma}}^{\dot{\gamma}}J_i))(0)\\
&=&D_{\dot{\gamma}}^{\dot{\gamma}}J_i(0)-\frac{g_{\dot\gamma}(D_{\dot{\gamma}}^{\dot{\gamma}}J_i(0),
\dot\gamma(0))}{g_{\dot\gamma}(\dot\gamma(0), \dot\gamma(0))}\dot\gamma(0)\\
&=&D_{\dot{\gamma}}^{\dot{\gamma}}J_i(0)=\tilde{S}^P_{\dot\gamma(0)}(e_i),\quad i=1,\cdots,k,\\
(D_{\dot{\gamma}}^{\dot{\gamma}}\textsf{nor}_\gamma(J_i))(0)&=&
\textsf{nor}_\gamma(D_{\dot{\gamma}}^{\dot{\gamma}}J_i))(0)\\
&=&D_{\dot{\gamma}}^{\dot{\gamma}}J_i(0)-\frac{g_{\dot\gamma}(D_{\dot{\gamma}}^{\dot{\gamma}}J_i(0),
\dot\gamma(0))}{g_{\dot\gamma}(\dot\gamma(0), \dot\gamma(0))}\dot\gamma(0)\\
&=&D_{\dot{\gamma}}^{\dot{\gamma}}J_i(0)=e_i,\quad i=k+1,\cdots,n-1.
\end{eqnarray*}
We obtain $\textsf{nor}_\gamma(J_i)=J_i$, $i=1,\cdots,n-1$. That is, $J_1,\cdots, J_{n-1}$ are $P$-Jacobi fields along $\gamma$
which are  perpendicular to $\dot\gamma$.
Suppose $\sum^{n}_{i=1}a_iJ_i=0$ for reals $a_1,\cdots,a_{n}$. Then
$\sum^{n}_{i=1}a_iJ_i(0)=0$ becomes $\sum^k_{i=1}a_ie_i=0$ by (\ref{e:Jacobibase2}). So $a_1=\cdots=a_k=0$.
Moreover
$$
0=D_{\dot{\gamma}}^{\dot{\gamma}}\left(\sum^{n}_{i=1}a_iJ_i\right)=\sum^{n}_{i=1}a_iD_{\dot{\gamma}}^{\dot{\gamma}}J_i\quad\hbox{and hence}\quad \sum^{n}_{i=1}a_iD_{\dot{\gamma}}^{\dot{\gamma}}J_i(0)=0,
$$
which implies $\sum^{n}_{i=k+1}a_ie_i=0$ by (\ref{e:Jacobibase2}) and therefore $a_i=0$ for $i=k+1,\cdots,n$.
Hence $J_1,\cdots, J_{n}$ form a basis of $\mathscr{J}^P_\gamma$.
\end{proof}

\begin{remark}\label{rm:Wu}
{\rm Since $J_1,\cdots, J_{n}$ are linearly independent,
if $\gamma(t_0)$ ($t_0\in (0,\tau]$) is not a $P$-focal point along $\gamma$,
then for any $(\lambda_1,\cdots,\lambda_{n})\in\mathbb{R}^{n}\setminus\{0\}$ it holds that
$\sum^{n}_{i=1}\lambda_iJ_i(t_0)\ne 0$.
Therefore $\{J_i(t_0)\}^{n-1}_{i=1}$ and $\{J_i(t_0)\}^{n}_{i=1}$ are bases for
 $(\dot\gamma^\bot)_{t_0}=\{v\in T_{\gamma(t_0)}M\,|\, g_{\dot{\gamma}(t_0)}(\dot\gamma(t_0),v)=0\}$
 and $T_{\gamma(t_0)}M$, respectively.
 }
 \end{remark}

\begin{lemma}\label{lem:Wu4.1}
Let $\gamma:[0,\tau]\to M$ be an  $L$-geodesic. (It must be $C^4$.)
 For an integer $1\le\ell\le 4$ let $V(t)$ ($0\le t\le \tau$) be a piecewise $C^\ell$ vector field along $\gamma$
 such that $V(0) = 0$. Then we have  a piecewise $C^{\ell-1}$-smooth
vector field $W$ along $\gamma$  with $W(0) = D_{\dot\gamma}^{\dot\gamma}V(0)$
and $W(t)=V(t)/t$ for $0< t\le\tau$.
\end{lemma}
\begin{proof}
Starting at any basis $e_1,\cdots,e_n$ for $T_{\gamma(0)}M$
  we may construct a $C^4$ parallel  orthonormal frame  fields $E_1,\cdots,E_n$ along $\gamma$. Write $V(t)=\sum_iv^i(t)E_i(t)$ ($0\le t\le\tau$),
where each $v^i$ is $C^\ell$ and satisfies $v^i(0)=0$. Clearly, we have a small $0<\varepsilon<\tau$ such that
$v^i$ is $C^\ell$ on $[0,\varepsilon]$. Then
  $v^i(t)=t\int^1_0(v^i)'(ts)ds$ for all $t\in [0,\varepsilon]$. Define
  $$
  w^i(t)=\left\{\begin{array}{ll}
  \int^1_0(v^i)'(ts)ds&\quad\hbox{if $t\in [0, \varepsilon/2]$},\\
  v^i(t)/t &\quad\hbox{if $t\in [\varepsilon/2, \tau]$}.
  \end{array}\right.
  $$
It is easily seen that  $w^i$ is piecewise $C^{\ell-1}$ and that
 $w^i(0)=(v^i)'(0)$ and $w^i(t)=v^i(t)/t$ for $0<t\le\tau$.
Set $W(t)=w^i(t)E_i(t)$. It is  piecewise $C^{\ell-1}$,  and $W(t)=V(t)/t$ for $0<t\le\tau$. Clearly,
$D_{\dot\gamma}^{\dot\gamma}V(t)=\sum_i({v}^i)'(t)E_i$ for $t\in [0,\tau]$,
and so $D_{\dot\gamma}^{\dot\gamma}V(0)=W(0)$.
\end{proof}

The following result of Riemannian geometry (see \cite[III. Lemma~2.9]{Sak96}, for example)
carry over to the pseudo-Finsler case immediately. (We shall prove it for completeness.)

\begin{lemma}\label{lem:Sa2.9}
Under the assumptions of Lemma~\ref{lem:base},
 $\tilde{\bf I}^\gamma_{P,q}$ is positive definite on $PC^{2}_{P\times q}(\gamma^\ast TM)$ provided that
$P$ has no focal points along $\gamma$ on $(0,\tau]$.
\end{lemma}

\begin{remark}\label{rm::Sa2.9}
{\rm As in (\ref{e:positive-focalA}) and (\ref{e:positive-focalB}),
Corollary~\ref{cor:Hermann} can only lead to that $\tilde{\bf I}^\gamma_{P,q}$ is positive definite
on the subspace $PC^{2}_{P\times q}(\gamma^\ast TM)_0$,
which is not sufficient for completing the proof of Proposition~\ref{prop:Sak3.1}.
However, if $P$ is a point, as in (\ref{e:positive-conjA})-(\ref{e:positive-conjB})
 Lemma~\ref{lem:Sa2.9} may be derived from Lemma~\ref{lem:31.9}.}
 \end{remark}

\begin{proof}[\bf Proof of Lemma~\ref{lem:Sa2.9}]
{\bf Step 1}.
Let $Y_1,\cdots, Y_n$ be a basis of $\mathscr{J}^P_\gamma$. (They are actually $C^4$.) We claim that
 for each $X\in PC^{2}_{P\times q}(\gamma^\ast TM)$
 there exist piecewise $C^1$ functions
$f^i(t)$ ($i=1,\cdots,n$) such that $X=\sum^n_{i=1}f^iY_i$.
Indeed, since $P$ has no focal points along $\gamma$ on $(0,\tau]$ we deduce
\begin{equation}\label{e:base}
\hbox{$\{Y_i(t)\}^{n}_{i=1}$ forms a basis for $T_{\gamma(t)}M$
for $0<t\le\tau$.}
\end{equation}

Consider  the basis of $\mathscr{J}^P_\gamma$ constructed in Lemma~\ref{lem:base},
 $J_1,\cdots, J_n$. (They are $C^4$.) For each $\alpha=k+1,\cdots,n$, since $J_\alpha(0)=0$, by the proof of Lemma~\ref{lem:Wu4.1}
   $$
  \tilde{J}_\alpha(t)=\left\{\begin{array}{ll}
  D_{\dot{\gamma}}^{\dot{\gamma}}J_\alpha(0)&\quad\hbox{if $t=0$},\\
  J_\alpha(t)/t &\quad\hbox{if $t\in (0, \tau]$}
  \end{array}\right.
  $$
 defines a $C^3$ vector field along $\gamma$.
 Recall that $\{J_i(0)=e_i\}^{k}_{i=1}$ and $\{D_{\dot{\gamma}}^{\dot{\gamma}}J_\alpha(0)=e_\alpha\}^n_{\alpha=k+1}$
form  a basis for $T_{\gamma(0)}P$ and  a basis for $(T_{\gamma(0)}P)^\bot_{\dot\gamma(0)}$, respectively.
From this and (\ref{e:base}) we obtain that
\begin{equation}\label{e:symmetric3}
\{J_j(t),\;\tilde{J}_\alpha(t)\,|\, 1\le j\le k,\;k+1\le\alpha\le n\}
\end{equation}
forms a basis of $T_{\gamma(t)}M$ for $0\le t\le \tau$.
Hence each $X\in PC^{2}_{P\times q}(\gamma^\ast TM)$ may be written as
$$
X(t)=\sum^k_{j=1}h^j(t)J_j(t)+ \sum^n_{\alpha=k+1}h^\alpha(t)\tilde{J}_\alpha(t),
$$
where $h^j(t)$ and $h^\alpha(t)$ are piecewise $C^2$ functions on $[0, \tau]$.
Since $X(0)\in T_{\gamma(0)}P$, we get $h^\alpha(0)=0$ and so $h^\alpha(t)=th^\alpha_1(t)$
with piecewise $C^1$ functions $h^\alpha_1$ by the proof of Lemma~\ref{lem:Wu4.1}. These lead to
$$
X(t)=\sum^k_{j=1}h^j(t)J_j(t)+ \sum^n_{\alpha=k+1}h^\alpha_1(t)J_\alpha(t)
$$
on $[0,\tau]$.
Note that $J_i=\sum^n_{j=1}a_{ij}Y_j$ for some matrix $(a_{ij})\in\mathbb{R}^{n\times n}$.
The claim at the beginning follows immediately.

\noindent{\bf Step 2}.
Let us take $Y_1,\cdots, Y_n$ as $J_1,\cdots, J_n$.
 Note that each $X\in PC^{2}_{P\times q}(\gamma^\ast TM)$ may be written as
 $X=\sum^n_{i=1}f^iJ_i$ for piecewise $C^1$ functions
$f^i(t)$ ($i=1,\cdots,n$). We obtain
\begin{eqnarray*}
g_{\dot\gamma}(D_{\dot{\gamma}}^{\dot{\gamma}}X(t), D_{\dot{\gamma}}^{\dot{\gamma}}X(t))
&=&\sum_{i,j}g_{\dot\gamma}(D_{\dot{\gamma}}^{\dot{\gamma}}(f^i(t)J_i(t)), D_{\dot{\gamma}}^{\dot{\gamma}}(f^j(t)J_j(t))\\
&=&\sum_{i,j}g_{\dot\gamma}((f^i(t))'J_i(t), (f^j(t))'J_j(t))\\
&&+\sum_{i,j}f^i(t)f^j(t)g_{\dot\gamma}(D_{\dot{\gamma}}^{\dot{\gamma}}(J_i(t)), D_{\dot{\gamma}}^{\dot{\gamma}}(J_j(t))\\
&&+\sum_{i,j}[(f^i(t))'f^j(t)+ (f^j(t))'f^i(t)]g_{\dot\gamma}(D_{\dot{\gamma}}^{\dot{\gamma}}(J_i(t)), J_j(t))
\end{eqnarray*}
because $g_{\dot\gamma}(D_{\dot{\gamma}}^{\dot{\gamma}}(J_i(t)), J_j(t))=g_{\dot\gamma}(J_i(t),
D_{\dot{\gamma}}^{\dot{\gamma}}(J_j(t)))$ by Lemma~\ref{lem:BoundaryBase}.
Moreover
\begin{eqnarray*}
g_{\dot\gamma}(D_{\dot{\gamma}}^{\dot{\gamma}}J_i, D_{\dot{\gamma}}^{\dot{\gamma}}J_j)&=&\frac{d}{dt}
g_{\dot\gamma}(D_{\dot{\gamma}}^{\dot{\gamma}}J_i, J_j)-
g_{\dot\gamma}(D_{\dot{\gamma}}^{\dot{\gamma}}D_{\dot{\gamma}}^{\dot{\gamma}}J_i, J_j)\nonumber\\
&=&\frac{d}{dt}g_{\dot\gamma}(D_{\dot{\gamma}}^{\dot{\gamma}}J_i, J_j)-
g_{\dot\gamma}(R_{\dot\gamma}(\dot\gamma,J_i)\dot\gamma, J_j)
\end{eqnarray*}
for $1\le i,j\le n$ by (\ref{e:compat}). Then
\begin{eqnarray*}
g_{\dot\gamma}(D_{\dot{\gamma}}^{\dot{\gamma}}X(t), D_{\dot{\gamma}}^{\dot{\gamma}}X(t))
&=&\sum_{i,j}g_{\dot\gamma}((f^i(t))'J_i(t), (f^j(t))'J_j(t))\\
&&-\sum_{i,j}f^i(t)f^j(t)g_{\dot\gamma}(R_{\dot\gamma}(\dot\gamma,J_i)\dot\gamma, J_j)\\
&&+\sum_{i,j}\frac{d}{dt}\left[f^i(t)f^j(t)g_{\dot\gamma}(D_{\dot{\gamma}}^{\dot{\gamma}}(J_i(t)), J_j(t))\right].
\end{eqnarray*}
and so
\begin{eqnarray*}
\int^\tau_0g_{\dot\gamma}(D_{\dot{\gamma}}^{\dot{\gamma}}X(t), D_{\dot{\gamma}}^{\dot{\gamma}}X(t))dt
&=&\sum_{i,j}\int^\tau_0g_{\dot\gamma}((f^i(t))'J_i(t), (f^j(t))'J_j(t))dt\\
&&-\sum_{i,j}\int^\tau_0f^i(t)f^j(t)g_{\dot\gamma}(R_{\dot\gamma}(\dot\gamma,J_i)\dot\gamma, J_j)dt\\
&&+\sum_{i,j}\left[f^i(t)f^j(t)g_{\dot\gamma}(D_{\dot{\gamma}}^{\dot{\gamma}}(J_i(t)), J_j(t))\right]^{t=\tau}_{t=0}.
\end{eqnarray*}
From this we derive
\begin{eqnarray*}
\tilde{\bf I}^\gamma_{P,q}(X, X)&=&
\int_0^\tau \left(g_{\dot\gamma}(R_{\dot\gamma}(\dot\gamma,X)\dot\gamma, X)+g_{\dot\gamma}(D_{\dot\gamma}^{\dot\gamma}X,D_{\dot\gamma}^{\dot\gamma}X)\right)dt
+g_{\dot\gamma(0)}(\tilde{S}^P_{\dot\gamma(0)}(X(0)), X(0))\\
&=&\sum_{i,j}\int_0^\tau g_{\dot\gamma}((f^i(t))'J_i(t), (f^j(t))'J_j(t))dt\\
&&+\sum_{i,j}\left[f^i(t)f^j(t)g_{\dot\gamma}(D_{\dot{\gamma}}^{\dot{\gamma}}(J_i(t)), J_j(t))\right]^{t=\tau}_{t=0}
+g_{\dot\gamma(0)}(\tilde{S}^P_{\dot\gamma(0)}(X(0)), X(0)).
 \end{eqnarray*}
Note that $0=X(\tau)=\sum_if^i(\tau)J_i(\tau)$ and that
$X(0)\in T_{\gamma(0)}P$ implies $f^i(0)=0$ for $i=k+1,\cdots,n$.
These and (\ref{e:Jacobibase1}) lead to
\begin{eqnarray*}
\sum_{i,j}\left[f^i(t)f^j(t)g_{\dot\gamma}(D_{\dot{\gamma}}^{\dot{\gamma}}(J_i(t)), J_j(t))\right]^{t=\tau}_{t=0}&=&-
\sum^k_{i=1}f^i(0)g_{\dot\gamma}(D_{\dot{\gamma}}^{\dot{\gamma}}J_i(0), X(0))\\
&=&-\sum^k_{i=1}f^i(0)g_{\dot\gamma}(\tilde{S}^P_{\dot\gamma(0)}(e_i), X(0))\\
&=&-g_{\dot\gamma}(\tilde{S}^P_{\dot\gamma(0)}(X(0)), X(0)).
 \end{eqnarray*}
Hence we arrive at
\begin{eqnarray}\label{e:symmetric5}
\tilde{\bf I}^\gamma_{P,q}(X, X)=\sum_{i,j}
\int_0^\tau g_{\dot\gamma}((f^i(t))'J_i(t), (f^j(t))'J_j(t))dt.
 \end{eqnarray}
 Since $g_{\dot\gamma(t)}$ is positive definite on $[0,\tau]$,
it follows that $\tilde{\bf I}^\gamma_{P,q}(X, X)\ge 0$ for each $X\in PC^{2}_{P\times q}(\gamma^\ast TM)$.
If $\tilde{\bf I}^\gamma_{P,q}(X, X)=0$ then for any $1\le i\le n$ we have
$(f^i(t))'\equiv 0$, and hence $f^i\equiv{\rm const}$.
Therefore $X$ is a $P$-Jacobi field, i.e., $X\in{\rm Ker}(\tilde{\bf I}^\gamma_{P,q})$.
But $\gamma(\tau)$ is not a focal point of $P$ along $\gamma$. Hence $X=0$, i.e., the desired conclusion is proved.
\end{proof}

\begin{proof}[\bf Proof of Proposition~\ref{prop:Sak3.1}]
Let  $0=t_0<\cdots<t_m=\tau$ be the  partition above Proposition~\ref{prop:Sak3.1}.
 Then $\gamma((0, t_1])$ contains  no $P$-focal points, and
 for each $i=1,\cdots, m-1$ and any $t_i\le a<b\le t_{i+1}$
points $\gamma(a)$ and $\gamma(b)$ are not mutually conjugate along $\gamma|_{[a, b]}$.
It follows that  ${\bf T}\cap PC^2(\gamma^\ast TM)_{00}=\{0\}$, where
$$
{\bf T}=\{Y\in PC^2_{P\times q}(\gamma^\ast TM)\,|\, Y(t_i)=0,\,i=1,\cdots,m-1\}.
$$
(Note that $Y(t_m)=0$ is clear.)
For $X\in PC^2_{P\times q}(\gamma^\ast TM)$,
by Remark~\ref{rm:Wu} we have a $C^2$ $P$-Jacobi field $Y_0$ along $\gamma|_{[0,t_1]}$ such that
$Y_0(t_1)=X(t_1)$, and $C^2$ Jacobi fields $Y_i$ along $\gamma|_{[t_i,t_{i+1}]}$ satisfying
$Y_i(t_i)=X(t_i)$ and $Y_i(t_{i+1})=X(t_{i+1})$ for $i=1,\cdots,m-1$.
Gluing these $Y_i$ yields a $Y\in PC^2_{P\times q}(\gamma^\ast TM)_{00}$
to satisfy the condition $Y(t_i)=X(t_i)$ ($i=1,\cdots,m-1$). Define $Z:=X-Y$. Then $Z(t_i)=0$ for $i=1,\cdots, t_m$
(since $X(t_m)=0=Y(t_m)$), that is, $Z\in {\bf T}$. Hence there is the following direct sum decomposition of vector spaces
$$
PC^2_{P\times q}(\gamma^\ast TM)={\bf T}\oplus PC^2_{P\times q}(\gamma^\ast TM)_{00}.
$$
\noindent{\bf Claim~C}. {\it ${\bf T}$ and $PC^2_{P\times q}(\gamma^\ast TM)_{00}$
are orthonormal with respect to $\tilde{\bf I}^\gamma_{P,q}$}.

In fact, for $W\in {\bf T}$ and
$V\in PC^2_{P\times q}(\gamma^\ast TM)_{00}$, since each $V|_{[t_i, t_{i+1}]}$ is $C^2$ and
\begin{eqnarray*}
g_{\dot\gamma}(D_{\dot{\gamma}}^{\dot{\gamma}}V, D_{\dot{\gamma}}^{\dot{\gamma}}W)=\frac{d}{dt}
g_{\dot\gamma}(D_{\dot{\gamma}}^{\dot{\gamma}}V, W)-
g_{\dot\gamma}(D_{\dot{\gamma}}^{\dot{\gamma}}D_{\dot{\gamma}}^{\dot{\gamma}}V, W)
\end{eqnarray*}
by (\ref{e:compat}),
as in  (\ref{e:secondDiff+}) we deduce
\begin{eqnarray*}
\tilde{\bf I}^\gamma_{P,q}(V, W)&=&
\sum^m_{i=0}\int_{t_i}^{t_{i+1}}\left(g_{\dot\gamma}(R_{\dot\gamma}(\dot\gamma, V)\dot\gamma, W)-g_{\dot\gamma}(D_{\dot\gamma}^{\dot\gamma}D_{\dot\gamma}^{\dot\gamma}V,W)\right)dt
\nonumber\\
&&+\sum^m_{i=0}g_{\dot\gamma}(D_{\dot{\gamma}}^{\dot{\gamma}}V, W)\Big|^{t_{i+1}}_{t_i}+g_{\dot\gamma(0)}(\tilde{S}^P_{\dot\gamma(0)}(V(0)), W(0))\\
&=&\int_0^\tau \left(g_{\dot\gamma}(R_{\dot\gamma}(\dot\gamma,V)\dot\gamma, W)-g_{\dot\gamma}(D_{\dot\gamma}^{\dot\gamma}D_{\dot\gamma}^{\dot\gamma}V,W)\right)dt
\nonumber\\
&&-g_{\dot\gamma(0)}(D_{\dot\gamma}^{\dot\gamma}V(0)-\tilde{S}^P_{\dot\gamma(0)}(V(0)),W(0))+
g_{\dot\gamma(1)}(D_{\dot\gamma}^{\dot\gamma}V(\tau),W(\tau))
\nonumber\\
&&+\sum^{m-1}_{i=1}g_{\dot\gamma(t)}\left(D_{\dot\gamma}^{\dot\gamma}V(t)\big|^{t_i+}_{t_i-},W(t_i)\right)=0.
 \end{eqnarray*}

\noindent{\bf Claim~D}. {\it $\tilde{\bf I}^\gamma_{P,q}$ is positive definite on ${\bf T}$}.

Indeed, let $\gamma_i=\gamma|_{[t_i,t_{i+1}]}$, $i=0,1,\cdots,m-1$.
If $W\in {\bf T}\setminus\{0\}$, then there is an $i$ such that $W|_{[t_i,t_{i+1}]}\ne 0$.
If $i=0$, then $\tilde{\bf I}^{\gamma_0}_{P,\gamma(t_1)}(W|_{[t_0,t_{1}]},W|_{[t_0,t_{1}]})>0$
 by Lemma~\ref{lem:Sa2.9}. (Note that (\ref{e:positive-focalA}) and (\ref{e:positive-focalB})
 are insufficient for proving this).   If $i>0$ then
  $\tilde{\bf I}^{\gamma_i}_{\gamma(t_i),\gamma(t_{i+1})}(W|_{[t_i,t_{i+1}]},W|_{[t_i,t_{i+1}]})>0$
  by (\ref{e:positive-conjA})-(\ref{e:positive-conjB}).
  It follows from these that
  $$
\tilde{\bf I}^\gamma_{P,q}(W,W)=\tilde{\bf I}^{\gamma_0}_{P,\gamma(t_1)}(W|_{[t_0,t_{1}]},W|_{[t_0,t_{1}]})+
\sum^{m-1}_{i=1}\tilde{\bf I}^{\gamma_i}_{\gamma(t_i),\gamma(t_{i+1})}(W|_{[t_i,t_{i+1}]},W|_{[t_i,t_{i+1}]})>0.
$$
Therefore we have proved Proposition~\ref{prop:Sak3.1}.
\end{proof}

\noindent{\bf Step 4}(\textsf{Prove the second part in} (iii)).
We only need to prove (\ref{e:MS2}).
 The ideas are the same as those of \cite[Theorem~1.2]{Pe06}. Let
$$
 PC^{2}_{P\times Q}(\gamma^\ast TM)=\{
\xi\in PC^{2}(\gamma^\ast TM)\,|\,\xi(0)\in T_{\gamma(0)}P,\;\xi(\tau))\in T_{\gamma(\tau)}Q\}.
$$
It is dense in $W^{1,2}_{P\times Q}(\gamma^\ast TM)$ and so
\begin{equation}\label{e:Step4}
{\rm Index}\left({\bf I}^\gamma_{P,Q}\right)={\rm Index}\left({\bf I}^\gamma_{P,Q}|_{PC^{2}_{P\times Q}(\gamma^\ast TM)}\right).
\end{equation}
Choose a complementary subspace $\mathscr{J}^P_{\gamma,0}$  of $\mathscr{J}^P_\gamma(\tau)$ in $\mathscr{J}^P_\gamma$,
that is, $\mathscr{J}^P_\gamma=\mathscr{J}^P_{\gamma,0}\oplus\mathscr{J}^P_\gamma(\tau)$.
By the assumptions of Theorem~\ref{th:MorseIndex}, $\{X(\tau)\,|\,X\in\mathscr{J}^P_\gamma\}\supseteq T_{\dot\gamma(\tau)}Q$.
(This may be satisfied if $\gamma(\tau)$ is not a $P$-focal point. In fact,
 $T_{\gamma(\tau)}Q\subset(\dot\gamma^\bot)_{\tau}=\{v\in T_{\gamma(\tau)}M\,|\, g_{\dot{\gamma}(\tau)}(\dot\gamma(\tau),v)=0\}$
and  $\{J_i(\tau)\}^{n-1}_{i=1}$ spans $ (\dot\gamma^\bot)_{\tau}$ by Remark~\ref{rm:Wu}.)
Then for each $V\in PC^{2}_{P\times Q}(\gamma^\ast TM)$ we have $W\in \mathscr{J}^P_{\gamma,0}$
  such that $W(\tau)=V(\tau)$.
  Since $Z:=V-W\in PC^{2}_{P\times q}(\gamma^\ast TM)$ and $PC^{2}_{P\times q}(\gamma^\ast TM)\cap\mathscr{J}^P_{\gamma,0}=\{0\}$ imply
$$
PC^{2}_{P\times Q}(\gamma^\ast TM)=PC^{2}_{P\times q}(\gamma^\ast TM)\oplus\mathscr{J}^{P}_{\gamma,0},
$$
by (\ref{e:secondDiff+}) it is easily seen that ${\bf I}^\gamma_{P,Q}(V, W)=0$ for any
$V\in \mathscr{J}^{P}_{\gamma,0}$ and $W\in PC^{2}_{P\times q}(\gamma^\ast TM)$. That is,
$PC^{2}_{P\times q}(\gamma^\ast TM)$ and $\mathscr{J}^{P}_{\gamma,0}$ are orthonormal with respect to
${\bf I}^\gamma_{P,Q}$. Hence
$$
{\rm Index}\left({\bf I}^\gamma_{P,Q}|_{PC^{2}_{P\times Q}(\gamma^\ast TM)}\right)={\rm Index}\left({\bf I}^\gamma_{P,Q}|_{PC^{2}_{P\times q}(\gamma^\ast TM)}\right)+ {\rm Index}\left({\bf I}^\gamma_{P,Q}|_{\mathscr{J}^P_{\gamma,0}}\right).
$$
Clearly, ${\bf I}^\gamma_{P,Q}|_{PC^{2}_{P\times q}(\gamma^\ast TM)}=\tilde{\bf I}^\gamma_{P,q}$ and
$$
{\bf I}^\gamma_{P,Q}|_{\mathscr{J}^{P}_{\gamma,0}}(J_1,J_2)=
g_{\dot\gamma(\tau)}(D_{\dot\gamma}^{\dot\gamma}J_1(\tau)-\tilde{S}^Q_{\dot\gamma(\tau)}(J_1(\tau)), J_2(\tau)).
$$
Note that $\mathscr{J}^P_\gamma(\tau)\subset {\rm Ker}\left({\bf I}^\gamma_{P,Q}|_{\mathscr{J}^{P}_{\gamma}}\right)$.
We have ${\rm Index}\left({\bf I}^\gamma_{P,Q}|_{\mathscr{J}^{P}_{\gamma,0}}\right)=
{\rm Index}\left({\bf I}^\gamma_{P,Q}|_{\mathscr{J}^{P}_{\gamma}}\right)$ and so
$$
{\rm Index}\left({\bf I}^\gamma_{P,Q}|_{PC^{2}_{P\times Q}(\gamma^\ast TM)}\right)={\rm Index}\left(
\tilde{\bf I}^\gamma_{P,q}\right)+ {\rm Index}\left({\bf I}^\gamma_{P,Q}|_{\mathscr{J}^P_{\gamma}}\right).
$$
(\ref{e:MS2}) follows from (\ref{e:Step4}) and this.
\hfill$\Box$\vspace{2mm}

\begin{proof}[\bf Proof of Proposition~B]
{\bf (i)}. In fact, for $X\in W^{1,2}_{P\times Q}(\gamma^\ast TM)$,  we obtain
$\textsf{tan}_\gamma(X)(0)=0$ and $\textsf{tan}_\gamma(X)(\tau)=0$
since $X(0)\in T_{\gamma(0)}P$, $X(\tau)\in T_{\gamma(\tau)}Q$
and $\gamma$ is $g_{\dot\gamma}$-orthonormal (or perpendicular) to $P$ and $Q$.
 It follows from these and  (\ref{e:secondDiff}) that
 \begin{eqnarray}\label{e:secondDiff-}
{\bf I}^\gamma_{P,Q}(\textsf{nor}_\gamma(V), \textsf{tan}_\gamma(W))&=&\int_0^\tau g_{\dot\gamma}(R_{\dot\gamma}(\dot\gamma,
\textsf{nor}_\gamma(V))\dot\gamma, \textsf{tan}_\gamma(W)dt\nonumber\\
&&+\int_0^\tau g_{\dot\gamma}(D_{\dot\gamma}^{\dot\gamma}\textsf{nor}_\gamma(V),D_{\dot\gamma}^{\dot\gamma}\textsf{tan}_\gamma(W))dt.
 \end{eqnarray}
 By Lemma~\ref{orthogJacobi}(i)
 $D_{\dot{\gamma}}^{\dot{\gamma}}\big(\textsf{tan}_{\gamma}(W)\big)=\textsf{tan}_{\gamma}(D_{\dot{\gamma}}^{\dot{\gamma}}W)$ and
$D_{\dot{\gamma}}^{\dot{\gamma}}\big(\textsf{nor}_\gamma(V)\big)=\textsf{nor}_\gamma(D_{\dot{\gamma}}^{\dot{\gamma}}V)$.
Hence the second term in the right side of (\ref{e:secondDiff-}) vanishes.
Moreover, by (\ref{e:Rsym8})
 $$
 g_{\dot\gamma}(R_{\dot\gamma}(\dot\gamma, \textsf{nor}_\gamma(V))\dot\gamma, \dot\gamma)=-g_{\dot\gamma}(R_{\dot\gamma}(\dot\gamma, \textsf{nor}_\gamma(V))\dot\gamma, \dot\gamma)
 $$
 and so $g_{\dot\gamma}(R_{\dot\gamma}(\dot\gamma, \textsf{nor}_\gamma(V))\dot\gamma, \dot\gamma)=0$, which implies
 \begin{eqnarray*}
 \textsf{tan}_\gamma\left(R_{\dot\gamma}(\dot\gamma, \textsf{nor}_\gamma(V))\dot\gamma\right)=
 \frac{g_{\dot\gamma}(R_{\dot\gamma}(\dot\gamma, \textsf{nor}_\gamma(V))\dot\gamma, \dot\gamma)}{g_{\dot\gamma}(\dot\gamma,\dot\gamma)}
\dot\gamma=0.
 \end{eqnarray*}
It follows that
$R_{\dot\gamma}(\dot\gamma, \textsf{nor}_\gamma(V))\dot\gamma=\textsf{nor}_\gamma\left(R_{\dot\gamma}(\dot\gamma, \textsf{nor}_\gamma(V))\dot\gamma\right)$
and hence
$$
g_{\dot\gamma}(R_{\dot\gamma}(\dot\gamma,
\textsf{nor}_\gamma(V))\dot\gamma, \textsf{tan}_\gamma(W)=0.
$$
{\bf (ii)}. By (i),  for all $V,W\in W^{1,2}_{P\times Q}(\gamma^\ast TM)$ it holds that
\begin{eqnarray}\label{e:secondDiff-*}
{\bf I}^\gamma_{P,Q}(V, W)={\bf I}^{\gamma,\bot}_{P,Q}(\textsf{nor}_\gamma(V), \textsf{nor}_\gamma(W))
+{\bf I}^\gamma_{P,Q}(\textsf{tan}_\gamma(V), \textsf{tan}_\gamma(W)).
\end{eqnarray}
If $V\in {\rm Ker}({\bf I}^{\gamma,\bot}_{P,Q})$, then $V=\textsf{nor}_\gamma(V)$ and $\textsf{tan}_\gamma(V)=0$.
It follows that $V\in {\rm Ker}({\bf I}^\gamma_{P,Q})$ and hence ${\rm Ker}({\bf I}^{\gamma,\bot}_{P,Q})\subset{\rm Ker}({\bf I}^\gamma_{P,Q})$.
Conversely, suppose $V\in {\rm Ker}({\bf I}^\gamma_{P,Q})$. (\ref{e:secondDiff-*}) leads to
\begin{eqnarray}\label{e:secondDiff-**}
{\bf I}^{\gamma,\bot}_{P,Q}(\textsf{nor}_\gamma(V), \textsf{nor}_\gamma(W))
+{\bf I}^\gamma_{P,Q}(\textsf{tan}_\gamma(V), \textsf{tan}_\gamma(W))=0,\quad\forall W\in W^{1,2}_{P\times Q}(\gamma^\ast TM).
\end{eqnarray}
Taking $W\in W^{1,2}_{P\times Q}(\gamma^\ast TM)^\bot$ so that $W=\textsf{nor}_\gamma(W)$ and $\textsf{tan}_\gamma(W)=0$, we derive from
(\ref{e:secondDiff-**}) that $\textsf{nor}_\gamma(V)\in {\rm Ker}({\bf I}^{\gamma,\bot}_{P,Q})$ and therefore
\begin{eqnarray}\label{e:secondDiff-***}
{\bf I}^\gamma_{P,Q}(\textsf{tan}_\gamma(V), \textsf{tan}_\gamma(W))=0,\quad\forall W\in W^{1,2}_{P\times Q}(\gamma^\ast TM).
\end{eqnarray}
Choosing $W=V$ we get
$$
0={\bf I}^\gamma_{P,Q}(\textsf{tan}_\gamma(V), \textsf{tan}_\gamma(V))
=\int^\tau_0\left(\frac{d}{dt}\frac{g_{\dot\gamma}(V,\dot\gamma)}{g_{\dot\gamma}(\dot\gamma,
\dot\gamma)}\right)^2g_{\dot\gamma}(\dot\gamma(t), \dot\gamma(t))dt=L(\dot\gamma(0))
\int^\tau_0\left(\frac{d}{dt}\frac{g_{\dot\gamma}(V,\dot\gamma)}{g_{\dot\gamma}(\dot\gamma,
\dot\gamma)}\right)^2dt
$$
because  $g_{\dot\gamma(t)}(\dot\gamma(t),\dot\gamma(t))=L(\dot\gamma(t))\equiv L(\dot\gamma(0))$ by \cite[Proposition~2.3]{Jav15}.
This implies that $t\mapsto g_{\dot\gamma(t)}(V(t),\dot\gamma(t))$ is constant.
Note that $g_{\dot\gamma(0)}(V(0),\dot\gamma(0))=0$ since $V(0)\in T_{\gamma(0)}P$ and $\dot\gamma(0)\in TP^\bot$.
We obtain that $g_{\dot\gamma(t)}(V(t),\dot\gamma(t))\equiv 0$ and so $V=\textsf{nor}_\gamma(V)\in {\rm Ker}({\bf I}^{\gamma,\bot}_{P,Q})$.
The first equality is proved.

In order to prove the second equality, observe that
${\rm Index}({\bf I}^{\gamma,\bot}_{P,Q})\le {\rm Index}({\bf I}^\gamma_{P,Q})$ by
the definition of Morse indexes. Moreover, we have known that both are finite.
Let $X_i$, $i=1,\cdots, m={\rm Index}({\bf I}^\gamma_{P,Q})$, be linearly independent vectors in
$W^{1,2}_{P\times Q}(\gamma^\ast TM)$ such that
${\bf I}^\gamma_{P,Q}$ is negative definite on the subspace generated by $X_1,\cdots,X_m$.
Then for any $(c_1,\cdots,c_m)\in\mathbb{R}^m\setminus\{0\}$ and $Y=\sum^m_{i=1}c_iX_i$ it holds that
\begin{eqnarray*}
0>{\bf I}^\gamma_{P,Q}(Y, Y)&=&{\bf I}^\gamma_{P,Q}(\textsf{tan}_\gamma(Y), \textsf{tan}_\gamma(Y))+
{\bf I}^\gamma_{P,Q}(\textsf{nor}_\gamma(Y), \textsf{nor}_\gamma(Y))\\
&=&{\bf I}^{\gamma,\bot}_{P,Q}\left(\sum^m_{i=1}c_i\textsf{nor}_\gamma(X_i), \sum^m_{i=1}c_i\textsf{nor}_\gamma(X_i)\right)+L(\dot\gamma(0))\int^\tau_0\left(\sum^m_{i=1}c_i\frac{d}{dt}\frac{g_{\dot\gamma}(X_i,\dot\gamma)}{g_{\dot\gamma}(\dot\gamma,
\dot\gamma)}\right)^2dt\\
&\ge&{\bf I}^{\gamma,\bot}_{P,Q}\left(\sum^m_{i=1}c_i\textsf{nor}_\gamma(X_i), \sum^m_{i=1}c_i\textsf{nor}_\gamma(X_i)\right).
\end{eqnarray*}
This shows that $\sum^m_{i=1}c_i\textsf{nor}_\gamma(X_i)\ne 0$ and therefore that
$\textsf{nor}_\gamma(X_1),\cdots,\textsf{nor}_\gamma(X_m)$ are also linearly independent.
It follows that ${\rm Index}({\bf I}^\gamma_{P,Q})\le {\rm Index}({\bf I}^{\gamma,\bot}_{P,Q})$.
\end{proof}

\section{Related conclusions and consequences, examples}\label{sec:3}

The following two supplement results of Lemma~\ref{lem:Sa2.9} will be helpful in applications.

\begin{proposition}\label{prop:Sa2.9}
Let $\gamma:[0,\tau]\to M$ be an $L$-geodesic
which  starts from $P$  perpendicularly.
Suppose that $g_{\dot\gamma(0)}$ is positive definite (and so each $g_{\dot\gamma(t)}$ is positive definite
and $L\circ\dot\gamma$ is a positive constant by Remark~A below (\ref{e:compat})).
\begin{description}
\item[(i)] If  $P$ has no focal point along $\gamma$ on $(0,\tau)$,
then $\tilde{\bf I}^\gamma_{P,q}$ is positive semi-definite on  $PC^{2}_{P\times q}(\gamma^\ast TM)$.
\item[(ii)] If there exists $s\in (0, \tau)$ such that $\gamma(s)$ is a focal point of $P$ along $\gamma|_{[0,s]}$,
then $\tilde{\bf I}^\gamma_{P,q}$ is negative definite on $PC^{2}_{P\times q}(\gamma^\ast TM)$.
Consequently, there exists a variation curve $\gamma_\eta$ of $\gamma$ in  $PC^2([0,\tau]; M, P, q)$ such that
for a sufficiently small $\eta>0$,
$\mathcal{E}_{P,q}(\gamma_\eta)<\mathcal{E}_{P,q}(\gamma)$, and so
$$
\mathcal{L}(\gamma)=\sqrt{2\tau\mathcal{E}_{P,q}(\gamma)}>\sqrt{2\tau\mathcal{E}_{P,q}(\gamma_\eta)}\ge \mathcal{L}(\gamma_\eta),
$$
where $\mathcal{L}(\alpha)=\int^\tau_0\sqrt{L(\dot\alpha(t))}dt$ for
$\alpha\in PC^2([0,\tau]; M, P, q)$ close to $\gamma$.
\end{description}
\end{proposition}
\begin{proof}
(i). Let $k=\dim\mathscr{J}^P_\gamma(\tau)$. Choose a basis of $\mathscr{J}^P_\gamma(\tau)$, $Y_1,\cdots,Y_k$.
Extend them into a basis of $\mathscr{J}^P_\gamma$, $Y_1,\cdots,Y_n$, such that $Y_{k+1}(\tau),\cdots, Y_n(\tau)$
are linearly independent. These $Y_i$ are $C^4$ actually. Let us prove that
 for each $X\in PC^{2}_{P\times q}(\gamma^\ast TM)$
 there exist piecewise $C^1$ functions $f^i(t)$ ($i=1,\cdots,n$) such that $X=\sum^n_{i=1}f^iY_i$.
By the proof of Step 1 of Lemma~\ref{lem:Sa2.9}, for any small $\varepsilon>0$
there exist piecewise $C^1$ functions $g^i(t)$ on $[0, \tau-\varepsilon/2]$ ($i=1,\cdots,n$)
such that
\begin{equation}\label{e:represent}
X(t)=\sum^n_{i=1}g^i(t)Y_i(t),\quad\forall t\in [0, \tau-\varepsilon/2].
\end{equation}
(Here the condition $L\circ\dot\gamma\ne 0$ is used.)
For any two $Y,Z\in \mathscr{J}^P_\gamma$, by Lemma~\ref{lem:BoundaryBase} we have
$g_{\dot\gamma}(D_{\dot\gamma}^{\dot\gamma}Y(t), Z(t))=g_{\dot\gamma}(Y(t), D_{\dot\gamma}^{\dot\gamma}Z(t))$
for all $t\in [0,\tau]$ and so
\begin{equation}\label{e:31.11}
g_{\dot\gamma}(D_{\dot\gamma}^{\dot\gamma}Y_i(\tau), Y_j(\tau))=g_{\dot\gamma}(D_{\dot\gamma}^{\dot\gamma}Y_j(\tau), Y_i(\tau))=0,\quad 1\le i\le k,\;k+1\le j\le n.
\end{equation}
Note that $\gamma$ and $Y_i$ are actually well-defined on a larger interval $[0, \tau+\epsilon)$ for some $\epsilon>0$.
If $\sum^k_{i=1}a_iD_{\dot\gamma}^{\dot\gamma}Y_i(\tau)=0$ for reals $a_1,\cdots,a_k$,
since $\sum^k_{i=1}a_iY_i(\tau)=0$ and $\sum^k_{i=1}a_iY_i$ is a $P$-Jacobi field it must hold that
$\sum^k_{i=1}a_iY_i(t)=0$ for all $t\in [0, \tau+\epsilon)$ by the ordinary differential equation theory,
which implies $a_1=\cdots=a_k=0$.
Hence $D_{\dot\gamma}^{\dot\gamma}Y_1(\tau),\cdots, D_{\dot\gamma}^{\dot\gamma}Y_k(\tau)$
are linearly independent. This fact and (\ref{e:31.11}) imply
$$
D_{\dot\gamma}^{\dot\gamma}Y_1(\tau),\cdots, D_{\dot\gamma}^{\dot\gamma}Y_k(\tau), Y_{k+1}(\tau),\cdots, Y_n(\tau)
$$
are linearly independent. Then for some small $\varepsilon>0$ and any $t\in [\tau-\varepsilon, \tau)$ the vectors
$$
\frac{1}{t-\tau}Y_i(t), Y_j(t),\quad 1\le i\le k,\;k+1\le j\le n,
$$
form a basis for $T_{\gamma(t)}M$. For each $i=1,\cdots,k$, define
$$
  \tilde{Y}_i(t)=\left\{\begin{array}{ll}
  D_{\dot{\gamma}}^{\dot{\gamma}}Y_i(\tau)&\quad\hbox{if $t=\tau$},\\
  Y_i(t)/(t-\tau) &\quad\hbox{if $t\in [\tau-\varepsilon, \tau)$}
  \end{array}\right.
  $$
 These are $C^3$ vector fields along $\gamma|_{[\tau-\varepsilon,\tau]}$ and
 $\{\tilde{Y}_i(t),\;{Y}_j(t)\,|\, 1\le i\le k,\;k+1\le j\le n\}$
forms a basis of $T_{\gamma(t)}M$ for $t\in [\tau-\varepsilon,\tau]$.
Hence we may write
$$
X(t)= \sum^k_{i=1}\tilde{h}^i(t)\tilde{Y}_i(t)+\sum^n_{j=k+1}h^j(t)Y_j(t)\quad\forall t\in[\tau-\varepsilon,\tau]
$$
where $\tilde{h}^i(t)$ and $h^j(t)$ are piecewise $C^2$ functions on $[\tau-\varepsilon, \tau]$.
Since $X(\tau)=0$ implies $\tilde{h}^i(\tau)=0$,
as in the proof of Lemma~\ref{lem:Wu4.1} we may get
piecewise $C^1$ functions $h^i$ such that $\tilde{h}^i(t)=(t-\tau)h^i(t)$
for $i=1,\cdots,k$. These lead to
$$
X(t)=\sum^k_{i=1} h^i(t)Y_i(t)+\sum^n_{j=k+1}h^j(t)Y_j(t),\quad\forall t\in [\tau-\varepsilon,\tau].
$$
It follows from this and (\ref{e:represent}) that $g^\ell=h^\ell$ on $[\tau-\varepsilon,\tau-\varepsilon/2]$
for $\ell=1,\cdots,n$ (because  $Y_1(t),\cdots,Y_n(t)$ form a basis of $T_{\gamma(t)}M$ for $0<t<\tau$).
Define
$$
  f^i(t)=\left\{\begin{array}{ll}
  g^i(t)&\quad\hbox{if $t\in [0,  \tau-\varepsilon]$},\\
  h^i(t) &\quad\hbox{if $t\in [\tau-\varepsilon, \tau]$}
  \end{array}\right.
  $$
for $i=1,\cdots,n$. They are piecewise $C^1$ on $[0,\tau]$ and satisfy $X(t)=\sum^n_{i=1}f^i(t)Y_i(t)$ for $t\in [0,\tau]$.

Let $J_1,\cdots,J_n$ be the basis of $\mathscr{J}^P_\gamma$ constructed in Lemma~\ref{lem:base}.
Then $J_i=\sum^n_{j=1}a_{ij}Y_j$ for some matrix $(a_{ij})\in\mathbb{R}^{n\times n}$.
Therefore each $X\in PC^{2}_{P\times q}(\gamma^\ast TM)$ may be written as
 $X=\sum^n_{i=1}f^iJ_i$ for piecewise $C^1$ functions $f^i(t)$ ($i=1,\cdots,n$).
 Since $g_{\dot\gamma(t)}$ is positive definite for each $t\in [0,\tau]$,
the desired conclusion may be derived as in the proof of Step 2 of Lemma~\ref{lem:Sa2.9}.\\

\noindent{(ii)}. The proof is similar to that of the Riemannian-Finsler case
(see $e.g$.  \cite[Theorem~1.12.13]{Kl95} and \cite[Proposition~4.3]{Wu14}).
Since $\gamma(s)$ is a focal point of $P$ along $\gamma|_{[0,s]}$, there is a nonzero $P$-Jacobi field
$J$ along $\gamma|_{[0,s]}$ such that $J(s)=0$. Note that $J$ belongs to
 ${\rm Ker}(\tilde{\bf I}^{\gamma_s}_{P,\gamma(s)})$ with $\gamma_s=\gamma|_{[0,s]}$ and therefore is $C^4$.
In particular, (\ref{e:P-JacobiField1}) implies $g_{\dot\gamma(t)}(\dot\gamma(t), J(t))=0\;\forall t\in [0, s]$.
As above we may assume that $J$ is actually well-defined on a larger interval $[0, \tau+\delta)$ for some $\delta>0$.
Since the unique $V$ satisfying the Jacobi equation (\ref{e:JacobiEq}) near $s\in [0, \tau+\delta)$
and $V(s)=0=D_{\dot\gamma}^{\dot\gamma}V(s)$ is $V=0$, we deduce that
 $D_{\dot\gamma}^{\dot\gamma}J(s)\ne 0$ and so $g_{\dot\gamma}(D_{\dot\gamma}^{\dot\gamma}J, D_{\dot\gamma}^{\dot\gamma}J)(s)\ne 0$.
   Define $Z(t)=J(t)$ if $t\in [0,  s]$, and $Z(t)=0$ if $t\in [s, \tau]$.
Then $Z\in PC^{2}_{P\times q}(\gamma^\ast TM)$ and $\tilde{\bf I}^\gamma_{P,q}(Z,Z)=0$.
Take a smooth function $\phi$ on $[0,\tau]$ such that (a) $\phi(s)=1$; (b) $\phi$ has small support about $s$
(so $\phi(0)=\phi(\tau)=0$);
and (c) $0\le\phi\le 1$.
Let $Y$ be the parallel transport of $-D_{\dot\gamma}^{\dot\gamma}J(s)$ along $\gamma$.
It is a $C^4$ parallel vector field along $\gamma$.  Set
$X_\epsilon=Z+\epsilon\phi Y$ for $\epsilon\in\R$.
Then $X_\epsilon\in PC^{2}_{P\times q}(\gamma^\ast TM)$ and
\begin{eqnarray*}
\tilde{\bf I}^\gamma_{P,q}(X_\epsilon, X_\epsilon)&=&\tilde{\bf I}^\gamma_{P,q}(Z,Z)+2\epsilon\tilde{\bf I}^\gamma_{P,q}(Z,\phi Y)
+\epsilon^2\tilde{\bf I}^\gamma_{P,q}(\phi Y,\phi Y)\\
&=&2\epsilon\int_0^s \left(g_{\dot\gamma}(R_{\dot\gamma}(\dot\gamma,Z)\dot\gamma, \phi Y)+g_{\dot\gamma}(D_{\dot\gamma}^{\dot\gamma}Z,D_{\dot\gamma}^{\dot\gamma}(\phi Y))\right)dt+
\epsilon^2\tilde{\bf I}^\gamma_{P,q}(\phi Y,\phi Y)\\
&=& 2\epsilon g_{\dot\gamma}(D_{\dot\gamma}^{\dot\gamma}Z,\phi Y)(s)+\epsilon^2\tilde{\bf I}^\gamma_{P,q}(\phi Y,\phi Y)\\
&=&-2\epsilon g_{\dot\gamma}(D_{\dot\gamma}^{\dot\gamma}J, D_{\dot\gamma}^{\dot\gamma}J)(s)+\epsilon^2\tilde{\bf I}^\gamma_{P,q}(\phi Y,\phi Y)
\end{eqnarray*}
because the integration in the first term in the second line is equal to
 \begin{eqnarray*}
 \int_0^s \left(g_{\dot\gamma}(R_{\dot\gamma}(\dot\gamma,Z)\dot\gamma, \phi Y)-g_{\dot\gamma}(D_{\dot\gamma}^{\dot\gamma}D_{\dot\gamma}^{\dot\gamma}Z,\phi Y)\right)dt
+g_{\dot\gamma(s)}(D_{\dot\gamma}^{\dot\gamma}Z(s), \phi(s)Y(s)).
 \end{eqnarray*}
It follows that $\tilde{\bf I}^\gamma_{P,q}(X_\epsilon, X_\epsilon)<0$ for small $\epsilon>0$
such that $\epsilon g_{\dot\gamma}(D_{\dot\gamma}^{\dot\gamma}J, D_{\dot\gamma}^{\dot\gamma}J)(s)>0$.

Fix such a $X_\epsilon$. Take variation curve $\gamma_\eta$ of $\gamma$ generated by $X_\epsilon$.
Then
$$
\frac{d}{d\eta}|_{\eta=0}\mathcal{E}_{P,q}(\gamma_\eta)=0\quad\hbox{and}\quad
\frac{d^2}{d\eta^2}|_{\eta=0}\mathcal{E}_{P,q}(\gamma_\eta)=\tilde{\bf I}^\gamma_{P,q}(X_\epsilon, X_\epsilon)<0.
$$
Hence for a sufficiently small $\eta>0$ it holds that $\mathcal{E}_{P,q}(\gamma_\eta)<\mathcal{E}_{P,q}(\gamma)$.
By \cite[Proposition~2.3]{Jav13}, $g_v(v,v)=L(v)\;\forall v\in A$, and hence
$L(\dot\gamma(t))=g_{\dot\gamma(t)}(\dot\gamma(t), \dot\gamma(t))>0$ for all $t\in [0, \tau]$
because $L\circ\dot\gamma$ is constant and $g_{\dot\gamma(t)}$ is positive definite for each $t\in [0,\tau]$.
It follows that  $g_v$ is positive definite
in a neighborhood of $\dot\gamma([0,\tau])$ in $A$.
Hence for sufficiently small $\eta\ge 0$ we deduce that
$L(\dot\gamma_\eta(t))=g_{\dot\gamma_\eta(t)}(\dot\gamma_\eta(t), \dot\gamma_\eta(t))>0$ for all $t\in [0, \tau]$.
This leads to
$$
\mathcal{L}(\gamma_\eta)=\int^\tau_0\sqrt{L(\dot\gamma_\eta(t))}dt\le\sqrt{\tau}\left(\int^\tau_0L(\dot\gamma_\eta(t))dt\right)^{1/2}
=\sqrt{2\tau\mathcal{E}_{P,q}(\gamma_\eta)}<\sqrt{2\tau\mathcal{E}_{P,q}(\gamma)}=\mathcal{L}(\gamma)
$$
because $L\circ\dot\gamma$ is constant.
\end{proof}

Let $PC^{2}_{P}(\gamma^\ast TM)=\{\xi\in PC^{2}(\gamma^\ast TM)\,|\,\xi(0)\in T_{\gamma(0)}P\}$.
 $\tilde{\bf I}^\gamma_{P,q}$ can be extended to a symmetric bilinear form
on $PC^{2}_{P}(\gamma^\ast TM)$, denoted by ${\bf I}^\gamma_{P}$.
We have the following pseudo-Finsler version of the so-called index lemma
in Riemannian-Finsler geometry (see $e.g$. \cite[III. Lemma~2.10]{Sak96}, \cite[Lemma~4.2]{Pe06} and \cite[Corollary~4.4]{Wu14}).
(The proof the Morse index theorem in \cite{Pe06} used \cite[Lemma~4.2]{Pe06}.)

\begin{proposition}\label{prop:Sa2.10}
Under the assumptions of Proposition~\ref{prop:Sa2.9},
suppose that  $P$ has no focal point along $\gamma$ on $(0,\tau]$.
Then for $X\in PC^{2}_{P}(\gamma^\ast TM)$ there exists a unique $P$-Jacobi field $Y\in \mathscr{J}^P_\gamma$ with
$Y(\tau)=X(\tau)$. Moreover $\tilde{\bf I}^\gamma_{P}(Y,Y)\le \tilde{\bf I}^\gamma_{P}(X,X)$,
where equality holds if and only if $Y=X$.
\end{proposition}
\begin{proof}
Since $\gamma(\tau)$ is not a $P$-focal point along $\gamma$ the first assertion holds.
By Step 2 in the proof of Lemma~\ref{lem:Sa2.9} we may write
 $X=\sum^n_{i=1}f^iJ_i$ for piecewise $C^1$ functions
$f^i(t)$ ($i=1,\cdots,n$). As the proof above (\ref{e:symmetric5}) we may obtain
\begin{eqnarray*}
\tilde{\bf I}^\gamma_{P}(X, X)
&=&\sum_{i,j}\int_0^\tau g_{\dot\gamma}((f^i(t))'J_i(t), (f^j(t))'J_j(t))dt\nonumber\\
&&+\sum_{i,j}\left[f^i(t)f^j(t)g_{\dot\gamma}(D_{\dot{\gamma}}^{\dot{\gamma}}(J_i(t)), J_j(t))\right]^{t=\tau}_{t=0}
+g_{\dot\gamma(0)}(\tilde{S}^P_{\dot\gamma(0)}(X(0)), X(0)),\\
&&\hspace{-2mm} g_{\dot\gamma(0)}(\tilde{S}^P_{\dot\gamma(0)}(X(0)), X(0))=
\sum_{i,j}f^i(0)f^j(0)g_{\dot\gamma(0)}(D_{\dot{\gamma}}^{\dot{\gamma}}J_i(0), J_j(0)).
\end{eqnarray*}
It follows that
\begin{eqnarray}\label{e:symmetric7}
\tilde{\bf I}^\gamma_{P}(X, X)
&=&\sum_{i,j}\int_0^\tau g_{\dot\gamma}((f^i(t))'J_i(t), (f^j(t))'J_j(t))dt\nonumber\\
&&+\sum_{i,j}f^i(\tau)f^j(\tau)g_{\dot\gamma}(D_{\dot{\gamma}}^{\dot{\gamma}}(J_i(\tau)), J_j(\tau)).
 \end{eqnarray}
Note that $\gamma(\tau)$ is not a $P$-focal point along $\gamma$ and that the $P$-Jacobi fields
$Y$ and $\sum_if^i(\tau)J_i$ along $\gamma$ coincide at $t=\tau$. We deduce that  $Y=\sum_if^i(\tau)J_i$ and therefore
$$
\sum_{i,j}f^i(\tau)f^j(\tau)g_{\dot\gamma(\tau)}(D_{\dot\gamma}^{\dot\gamma}J_i(\tau), J_j(\tau))=
g_{\dot\gamma(\tau)}(D_{\dot\gamma}^{\dot\gamma}Y(\tau), Y(\tau))
$$
Moreover, $g_{\dot\gamma(t)}$ is positive definite for each $t\in [0,\tau]$.
Then (\ref{e:symmetric7}) yields
$$
\tilde{\bf I}^\gamma_{P}(X, X)\ge g_{\dot\gamma(\tau)}(D_{\dot\gamma}^{\dot\gamma}Y(\tau), Y(\tau))=\tilde{\bf I}^\gamma_{P}(Y, Y),
$$
and equality holds if and only if $(f^i(t))'J_i(t)\equiv 0\;\forall i$. The latter means that $X=Y$.
\end{proof}

If $L=F^2$ for some $C^6$ Finsler metric $F$ on $M$,
as a consequence of Theorem~\ref{th:MorseIndex} and
splitting lemmas in \cite{Lu5} we may obtain the following generalization of \cite[Theorem~2.5.10]{Kl95},
which strengthens Proposition~\ref{prop:Sa2.9}.

\begin{theorem}\label{th:Kl2.5.10}
Let $M$ be an $n$-dimensional, connected $C^7$ submanifold of $\R^N$ of dimension greater than $1$, and
let $P$ be a $C^7$ submanifold of $M$ of dimension less than $\dim M$.
For a $C^6$ Finsler metric $F$ on $M$  let
$\gamma:[0,\tau]\to M$ be a (nonzero) constant speed $F$-geodesic which is perpendicular to $P$,
that is,   a critical point of the functional
$$
\alpha\mapsto\mathcal{E}_{P,q}(\alpha)=\frac{1}{2}\int^\tau_0[F(\alpha(t),\dot\alpha(t))]^2dt
$$
on $W^{1,2}([0,\tau]; M, P, q):=\{\alpha\in H^{1}([0, \tau]; M)\,|\,\alpha(0)\in P, \alpha(\tau)=q\}$ with $q=\gamma(\tau)$.
 \begin{description}
\item[(i)] If there are no $P$-focal points on $\gamma((0, \tau])$ then there
exists a neighborhood $\mathscr{U}(\gamma)$ of $\gamma$ in $W^{1,2}([0,\tau]; M, P, q)$ such that for all
$\alpha\in\mathscr{U}(\gamma)$, $\mathcal{E}_{P,q}(\alpha)\ge \mathcal{E}_{P,q}(\gamma)$, with
$\mathcal{E}_{P,q}(\alpha) =\mathcal{E}_{P,q}(\gamma)$ only for $\alpha=\gamma$.

\item[(ii)] Let $m > 0$ be the number of $P$-focal points on $\gamma((0, \tau))$, each counted with
its multiplicity. Then there exists a topological embedding
$$
\varphi:B^m\to W^{1,2}([0,\tau]; M, P, q)
$$
of the unit ball $B^m =\{x\in\R^m\,|\, |x|<1\}$ such that $\varphi(0)=\gamma$,
$\mathcal{E}_{P,q}(\varphi(x))\le \mathcal{E}_{P,q}(\gamma)$ and hence $L(\varphi(x))\le L(\gamma)$
for all $x\in B^m$, with equality only for $x=0$. Here $L(\alpha)=\int^\tau_0F(\alpha(t),\dot\alpha(t))dt$.

\item[(iii)] If $\gamma$ is also perpendicular to a submanifold $Q\subset M$ at $q=\gamma(\tau)$
and $\{X(\tau)\,|\,X\in\mathscr{J}^P_\gamma\}\supseteq T_{\dot\gamma(\tau)}Q$,
then for $m>0$ as in (ii) there exists a topological embedding
$$
\varphi:B^m\to W^{1,2}([0,\tau]; M, P, Q)
$$
such that $\varphi(0)=\gamma$,
$\mathcal{E}_{P, Q}(\varphi(x))\le \mathcal{E}_{P, Q}(\gamma)$ and hence $L(\varphi(x))\le L(\gamma)$,
for all $x\in B^m$, with equality only for $x=0$.
\end{description}
\end{theorem}
\begin{proof}
In the present case, we choose a $C^6$ Riemannian metric $h$ on $M$ such that
$P$ and $Q$ are totally geodesic near $\gamma(0)$ and $\gamma(\tau)$, respectively.
Then using it we define the norm in (\ref{e:1.1}).
Let
 $$
 {\bf B}_{2\rho}(T_{\gamma}\Lambda_{P\times q}(M)):= \{\xi\in
T_{\gamma}\Lambda_{P\times q}(M)\,|\,\|\xi\|_{1}<2\rho\}
$$
for $\rho>0$. If $\rho>0$ is small enough, the map
\begin{equation}\label{e:1.8}
{\rm EXP}_{\gamma}: {\bf
B}_{2\rho}(T_{\gamma}\Lambda_{P\times q}(M))\to\Lambda_{P\times q}(M)
\end{equation}
given by ${\rm EXP}_{\gamma}(\xi)(t)=\exp_{\gamma(t)}(\xi(t))$,
is only  a $C^{2}$ coordinate chart around $\gamma$ on $\Lambda_{P\times q}(M)$, where $\exp$ is the exponential map
defined by the Riemannian metric $h$. The restriction $\mathcal{E}^X_{P,q}$ of $\mathcal{E}_{P,q}$ to the Banach manifold
$$
C^1([0,\tau]; M, P, q):=\{\alpha\in C^1([0, \tau]; M)\,|\,\alpha(0)\in P, \alpha(\tau)=q\}
$$
with $q=\gamma(\tau)$
is $C^2$ near $\gamma$, and the Hessian $D^2\mathcal{E}^X_{P,q}(\gamma)$
(a continuous symmetric bilinear form on $T_\gamma C^1([0,\tau]; M, P, q)$)
 can be extended into  a continuous  symmetric bilinear form on $W^{1,2}_{P\times q}(\gamma^\ast TM)=
 T_{\gamma}\Lambda_{P\times q}(M)$
by a localization argument (cf. \cite{Lu4,Lu5}), which  is exactly equal to
 ${\bf I}^\gamma_{P,q}$ defined below  (\ref{e:secondDiff*}) by uniqueness.
 Our localization arguments (\cite[Sect.4]{Lu4}) imply that there exists a bounded linear self-adjoint Fredholm
 operator $B$ on the Hilbert space $W^{1,2}_{P\times q}(\gamma^\ast TM)$ such that
 ${\bf I}^\gamma_{P,q}(\xi,\eta)=\langle B\xi,\eta\rangle_1$ for all $\xi,\eta\in W^{1,2}_{P\times q}(\gamma^\ast TM)$.
(Actually, we showed that $B$ is the sum of a positive definite operator and a compact linear self-adjoint operator.
This can also be seen from (\ref{e:secondDifff4}). )
By \cite[Proposition~4.5]{Con} or \cite[Proposition~B.3]{Lu3}, $0$ is at most an isolated point of the spectrum $\sigma(B)$ of $B$.
Therefore we can choose $\epsilon>0$ such that $[-2\epsilon, 2\epsilon]\cap\sigma(B)=\{0\}$.
It follows from this and (the proof of) \cite[Theorem~7.1]{Hes51}
that there exists a unique direct sum  decomposition of the space
\begin{equation}\label{e:1.9}
T_{\gamma}\Lambda_{P\times q}(M)={\bf H}^-({\bf I}^\gamma_{P,q})\oplus {\bf H}^0({\bf I}^\gamma_{P,q}
)\oplus{\bf H}^+({\bf I}^\gamma_{P,q})
\end{equation}
having the following properties:
\begin{description}
\item[(a)] The spaces ${\bf H}^-({\bf I}^\gamma_{P,q})$, ${\bf H}^0({\bf I}^\gamma_{P,q})$ and ${\bf H}^+({\bf I}^\gamma_{P,q})$
are mutually orthonormal with respect to both ${\bf I}^\gamma_{P,q}$ and $\langle\cdot,\cdot\rangle_1$.
\item[(b)] ${\bf I}^\gamma_{P,q}$ (or $B$) is negative definite on
${\bf H}^-({\bf I}^\gamma_{P,q})$, zero on ${\bf H}^0({\bf I}^\gamma_{P,q})$, and positive definite on ${\bf H}^+({\bf I}^\gamma_{P,q})$.
\end{description}
Let  $P^+$, $P^-$ and $P^0$ denote projection operators onto
${\bf H}^+({\bf I}^\gamma_{P,q})$, ${\bf H}^-({\bf I}^\gamma_{P,q})$ and ${\bf H}^0({\bf I}^\gamma_{P,q})$,
respectively. Then they satisfy $P^0P^-=P^0P^+=P^-P^+=0$, $P^\ast B=BP^\ast,\;\ast=+,-,0$, and
 $P^++P^0+P^-$ is equal to the identity mapping on $T_{\gamma}\Lambda_{P\times q}(M)$.

By \cite[Theorem~1.1]{Lu5} we get $\delta\in (0, 2\rho]$
and a unique $C^{2}$-map
$\chi: {\bf B}_\delta\bigl({\bf H}^0({\bf I}^\gamma_{P,q})\bigr)\to {\bf H}^-({\bf I}^\gamma_{P,q}
 )\oplus{\bf H}^+({\bf I}^\gamma_{P,q})$
satisfying  $\chi(0)=0$ and $d\chi(0)=0$,  and  an origin-preserving
homeomorphism $\psi$ from ${\bf B}_\delta(T_{\gamma}\Lambda_{P\times q}(M))$
to an open neighborhood of $0$ in $T_{\gamma}\Lambda_{P\times q}(M)$
 such that
\begin{equation}\label{e:1.10}
{\cal E}_{P,q}\circ{\rm EXP}_{\gamma}\circ\psi(\xi)=\|P^+\xi\|_1^2-
\|P^-\xi\|_1^2 + {\cal E}_{P,q}\circ{\rm EXP}_{\gamma}(P^0\xi+ \chi(P^0\xi))
\end{equation}
for all $\xi\in{\bf B}_\delta(T_{\gamma}\Lambda_{P\times q}(M))$.

In case (i), Theorem~\ref{th:MorseIndex} and (\ref{e:P-JacobiField1}) imply, respectively,
${\bf H}^-({\bf I}^\gamma_{P,q})=\{0\}$ and ${\bf H}^0({\bf I}^\gamma_{P,q})=\{0\}$.
Then (\ref{e:1.10}) becomes
\begin{equation*}
{\cal E}_{P,q}\circ{\rm EXP}_{\gamma}\circ\psi(\xi)=\|\xi\|_1^2
+ {\cal E}(\gamma),\quad\forall \xi\in{\bf B}_\delta(T_{\gamma}\Lambda_{P\times q}(M)).
\end{equation*}
It follows that $\mathscr{U}(\gamma):={\rm EXP}_{\gamma}\circ\psi\left({\bf B}_\delta(T_{\gamma}\Lambda_{P\times q}(M))\right)$
is a desired neighborhood of $\gamma$.

In case (ii), by Theorem~\ref{th:MorseIndex} we have
$\dim{\bf H}^-({\bf I}^\gamma_{P,q})=m$.
Let $\xi_1,\cdots,\xi_m$ be a unit orthonormal basis in ${\bf H}^-({\bf I}^\gamma_{P,q})$. Define
$\varphi:B^m\to W^{1,2}([0,\tau]; M, P, q)$ by
$$
\varphi(x_1,\cdots,x_m)={\rm EXP}_{\gamma}\circ\psi\left(\delta\sum^m_{i=1}x_i\xi_i\right).
$$
This is a topological embedding satisfying $\varphi(0)=\gamma$. By (\ref{e:1.10}) we have
\begin{equation*}
{\cal E}_{P,q}\big(\varphi(x_1,\cdots,x_m)\big)=-\delta^2\sum^m_{i=1}x_i^2 + {\cal E}_{P,q}({\gamma}),\quad
\forall (x_1,\cdots,x_m)\in B^m.
\end{equation*}
Note that $L(\varphi(x))^2\le 2\tau\mathcal{E}_{P,q}(\varphi(x))\le 2\tau\mathcal{E}_{P,q}(\gamma)=L(\gamma)^2$.
The expected assertions follow from these immediately.

In order to get (iii),
by shrinking $\rho>0$ (if necessary) we have a $C^2$ coordinate chart around $\gamma$ on $\Lambda_{P\times Q}(M)$,
\begin{equation*}
{\rm EXP}_{\gamma}: {\bf
B}_{2\rho}(T_{\gamma}\Lambda_{P\times Q}(M))\to\Lambda_{P\times Q}(M)
\end{equation*}
given by ${\rm EXP}_{\gamma}(\xi)(t)=\exp_{\gamma(t)}(\xi(t))$.
Then by \cite[Theorem~1.1]{Lu5} we get $\delta\in (0, 2\rho]$
and a unique $C^{2}$-map
$\chi: {\bf B}_\delta\bigl({\bf H}^0({\bf I}^\gamma_{P,Q})\bigr)\to {\bf H}^-({\bf I}^\gamma_{P,Q}
 )\oplus{\bf H}^+({\bf I}^\gamma_{P,Q})$
satisfying  $\chi(0)=0$ and $\chi(0)=0$,  and  an origin-preserving
homeomorphism $\psi$ from ${\bf B}_\delta(T_{\gamma_0}\Lambda_{P\times Q}(M))$
to an open neighborhood of $0$ in $T_{\gamma}\Lambda_{P\times Q}(M)$
 such that
\begin{equation*}
{\cal E}_{P,Q}\circ{\rm EXP}_{\gamma}\circ\psi(\xi)=\|P^+\xi\|_1^2-
\|P^-\xi\|_1^2 + {\cal E}_{P,Q}\circ{\rm EXP}_{\gamma}(P^0\xi+ \chi(P^0\xi))
\end{equation*}
for all $\xi\in{\bf B}_\delta(T_{\gamma}\Lambda_{P\times Q}(M))$. Since $\dim{\bf H}^-({\bf I}^\gamma_{P,Q})\ge m$ by (\ref{e:MS2}),
repeating the proof of (ii) we may obtain the assertions.
(Clearly, $m$ may be replaced by $m+ {\rm Index}(\mathcal{A}_\gamma)$. )
\end{proof}

Finally, we discuss the exponential map of a $C^6$ conic Finsler manifold
$(M,L)$  with domain $A\subset TM\setminus 0_{TM}$.
For $p\in M$, let $\mathcal{D}_p$ be the set of vectors $v$ in $A\cap T_pM$
such that the unique maximal geodesic $\gamma_v:[0, b_v)\to M$ satisfying $\gamma_v(0)=p$ and $\dot\gamma_v(0)=v$, where $b_v>1$.
In the present case, by the proof of \cite[Prop.2.15]{Jav15} we see that
$\mathcal{D}_p$ is an open subset of $A\mathop\cap T_pM$ with property that
$v\in \mathcal{D}_p$ implies $tv\in \mathcal{D}_p$ for all $0<t\le 1$.
Define the \textsf{exponential map} of $(M,L)$ at $p$   by
 $$
  \exp^L_p:\mathcal D_p\to M,\quad v\mapsto\exp^L_p(v):=\gamma_v(1).
 $$
It is $C^3$.
When $A\cap T_pM=T_pM\setminus\{0\}$ the map $\exp^L_p$
can be extended to a $C^1$ map in an open neighborhood of $0_p$ by putting $\exp^L_p(0_p)=p$
(\cite[Prop.2.15]{Jav15}). Moreover, for $v\in\mathcal{D}_p$ and $w\in T_v(T_pM)$ it holds that
$D\exp^L_p(v)[w]=J(1)$, where $J$ is the unique Jacobi field along $\gamma_v$ such that $J(0)=0$ and $J'(0)=w$
(\cite[Prop.3.15]{Jav15}).

Let $P\subset M$ be a $C^7$ submanifold of dimension $k<n$ as before.
Since $g_{\lambda v}=g_v$ for any $v\in A$ and $\lambda>0$, by (\ref{e:normalBundle}) we see that
$TP^\bot$ has the cone property, i.e., $v\in TP^\bot$ implies $\lambda v\in TP^\bot$ for all $\lambda>0$.
Let $\mathcal{D}=\cup_{p\in M}\mathcal{D}_p$. It is an open subset in $A$ and
$\mathcal{D}\cap TP^\bot$ is also an $n$-dimensional submanifold
of $TM$.  For each $v\in TP^\bot$, since $tv\in \mathcal{D}\cap TP^\bot$ for small $t>0$,
we have $\pi(\mathcal{D}\cap TP^\bot)=\pi(TP^\bot)=P_0$. Therefore the map
$\pi:\mathcal{D}\cap TP^\perp\rightarrow P_0$ is a submersion.
Let
$$
(\mathcal{D}\cap TP^\bot)^\ast=\left\{v\in \mathcal{D}\cap TP^\bot\;\big|\; \hbox{$g_{v}|_{T_{\pi(v)}P\times T_{\pi(v)}P}$ is nondegenerate}\right\}.
$$
This is an open subset of $\mathcal{D}\cap TP^\bot$, and
  $v\in(\mathcal{D}\cap TP^\bot)^\ast$ implies
$tv\in (\mathcal{D}\cap TP^\bot)^\ast$ for all $0<t\le 1$.
Denote by  $\exp^{LN}$ the restriction of $\exp^L$ to $(\mathcal{D}\cap TP^\bot)^\ast$.

For $v\in (\mathcal{D}\cap TP^\bot)^\ast$
the unique maximal geodesic $\gamma_v$ is orthogonal to $P$ at $\gamma_v(0)=\pi(v)$.
Let $p=\pi(v)\in P_0$, $u\in T_v(\mathcal{D}\cap TP^\bot)^\ast$  and let
 $\alpha:(-\epsilon,\epsilon)\to(\mathcal{D}\cap TP^\bot)^\ast$ be a $C^4$ curve such that $\alpha(0)=v$ and $\alpha'(0)=u$.
Put $c=\pi\circ\alpha$. Then $\alpha(s)\in \mathcal{D}_{c(s)}\cap T_{c(s)}P^\bot$ for $s\in (-\epsilon,\epsilon)$.
 Define
$$
\Lambda(t,s)=\exp^{LN}(t\alpha(s))=\gamma_{\alpha(s)}(t),\quad\forall (t,s)\in [0, 1]\times(-\epsilon,\epsilon).
$$
Then $\Lambda(t,0)=\gamma_v(t)$,   $\Lambda(0,s)=c(s)\in P_0$ and $\partial_t\Lambda(0,s)=\alpha(s)$ for all $s$.

\begin{proposition}\label{prop:ExponJ}
$t\mapsto J(t)=\partial_s\Lambda(t,0)$ is a $P$-Jacobi field along $\gamma_v$,
and $D\exp^{LN}(v)[u]=J(1)$.
\end{proposition}
\begin{proof}
By \cite[Prop.~3.13]{Jav15}, $J$  is a Jacobi field along $\gamma_v$.
$J(1)=\frac{d}{ds}\Lambda(1,s)|_{s=0}=D\exp^{LN}(v)[u]$.
By the proof above \cite[(4.4)]{Wu14} we have
${\rm tan}^P_{\dot\gamma_v(0)}\big(D_{\dot\gamma_v}^{\dot\gamma_v}J(0))=\tilde{S}^P_{\dot\gamma_v(0)}(J(0))$,
namely, $J$ is a $P$-Jacobi field.
\end{proof}

For $0<\tau\le 1$ let $\Lambda_\tau(t,s)=\exp^{LN}(t\tau\alpha(s))=\Lambda(\tau t,s)$. Denote by
$\tau\alpha$ the curve in $(\mathcal{D}\cap TP^\bot)^\ast$ given by $s\mapsto \tau\alpha(s)$.
Then  $\tau\alpha(0)=\tau v$, $(\tau\alpha)'(0)=\tau u$ and
$$
D\exp^{LN}(\tau v)[\tau u]=\frac{\partial}{\partial s}\Lambda_\tau(1,s)|_{s=0}=
\frac{\partial}{\partial s}\Lambda(\tau1,s)|_{s=0}=J(\tau).
$$

By the conclusion (i) of Theorem~\ref{th:MorseIndex}, if $s>0$ is small enough, $\gamma_v(s)$
is not a $P$-focal point along $\gamma_v|_{[0,s]}$, and therefore $D\exp^{LN}(\tau v)$
is an isomorphism because $\dim(\mathcal{D}\cap TP^\bot)^\ast=\dim M=n$.

Note that for $v\in (\mathcal{D}\cap TP^\bot)^\ast$ there exists $a_v>1$ such that $a_vv\notin
(\mathcal{D}\cap TP^\bot)^\ast$ and that $ta_vv\in(\mathcal{D}\cap TP^\bot)^\ast$ for all $0<t<1$.
By the above arguments there exists $0<d_v\le a_v$ such that $d_vv$ depends only on
the ray $[v]:=\{tv\,|\, t>0\}$ and that $D\exp^{LN}(tv)$
is an isomorphism for each $0<t<d_v$. Set
$(\mathcal{D}\cap TP^\bot)^{\ast\ast}:=\{tv\,|\, v\in (\mathcal{D}\cap TP^\bot)^\ast,\;0<t<d_v\}$.
We have
\begin{proposition}\label{prop:ExponJ1}
For each $v\in (\mathcal{D}\cap TP^\bot)^{\ast\ast}$, the restriction of $\exp^{LN}$
to a small neighborhood of $v$ in $\mathcal{D}\cap TP^\bot$ is
a local $C^3$ diffeomorphism.
\end{proposition}

Specially, if $P$ is a point $p$, by the conclusion (ii) of Theorem~\ref{th:MorseIndex} we may deduce

\begin{proposition}\label{prop:exponential}
For each $v\in\mathcal{D}_p$, there exists $d_v>0$ such that
$td_vv\in\mathcal{D}_p$ for all $t\in (0, 1]$,
$d_vv$ depends only on the ray $[v]:=\{tv\,|\, t>0\}$ and that the restriction of $\exp^{L}$
to a small neighborhood of $tv$ in $\mathcal{D}_p$ is $C^3$ diffeomorphism.
\end{proposition}

If $L=F^2$ for some conic Finsler metric $F$ on $M$ there is a better result, see
\cite[Proposition~3.21]{Jav14}.

Using the theory developed in this paper
it is possible  to give generalizations of related results in Riemannian-Finsler geometry.
They will be studied elsewhere.\\

\noindent{\bf Examples}.
 There exist many examples of (conic) pseudo-Finsler metrics and conic Finsler metrics (cf. \cite{CaJaSa14} and \cite{Jav14, Jav20S}).
A popular conic Finsler metric on $M$ is the Kropina metric $K$ given by
$K(v)=-\frac{h(v,v)}{2\omega(v)}$, where $h$ is a Riemannian metric on $M$ and $\omega$ is a nonvanishing $1$-form on $M$.
$K$ is defined on $TM\setminus{\rm Ker}(\omega)$, and usually it is restricted to
$A=\{v\in TM\,|\,\omega(v)<0\}$,  on which we have $K(v) > 0$. The pair $(M, K)$ is often called a  \textsf{Kropina space},
and is related to the Zermelo's navigation problem which consists in finding the paths
  between two points $p$ and $q$ that minimize the travel time of a ship or an airship moving in a wind in a Riemannian landscape
  $(M, h)$.   Recently, Cheng et al. gave a variational characterization of chains in CR geometry as geodesics of a certain Kropina metric
 (\cite{ChMaMa}). Caponio et al. studied the existence of connecting and closed geodesics in a manifold endowed with a Kropina metric
   (\cite{CaGiMaSu}).

 \begin{theorem}[\hbox{\cite[Theorem~4.4]{CaGiMaSu}}]\label{th:4.1}
Let $(M, K)$ be above and let $p$, $q$ be two different points in $M$ such that
$\Omega_{p,q}(A)$  the set of the piecewise smooth, $A$-admissible curves between $p$ and $q$ is nonempty.
Assume that $(M, h)$ is complete and there exists a point $\bar{x}$ and a positive constant $C_{\bar{x}}$ such that
$\|\omega\|_x\le C_{\bar{x}}(d_h(x,\bar{x})+1)$. Then for each connected component $\mathscr{C}$ of $W^{1,2}([0,1]; M,p,q)$
there exists a geodesic of the Kropina space from $p$ to $q$, which is a minimizer of the Kropina length functional on $\mathscr{C}$,
provided that there exists an $A$-admissible curve in $\mathscr{C}$.
 Moreover, if $\mathscr{C}$ corresponds to a non-trivial element of the fundamental group of $M$, a
geodesic loop in $\mathscr{C}$ exists when $p$ and $q$ coincide.
\end{theorem}

Let $\gamma:[0, 1]\to M$ be the geodesic of the Kropina space $(M, K)$ from $p$ to $q$ given by Theorem~\ref{th:4.1}.
By Theorem~\ref{th:MorseIndex} and (\ref{e:MS5}), $p$ has only finitely many conjugate points along $\gamma$ and
\begin{equation*}
{\rm Index}(\tilde{\bf I}^\gamma_{p,q})={\rm Index}({\bf I}^\gamma_{p,q})=\sum\limits_{t_{0}\in(0,\tau)}\mu^{p}(t_{0}).
\end{equation*}
We conclude that there is no $s\in (0, 1)$ such that $\gamma(s)$ is a conjugate point of $p$ along $\gamma|_{[0,s]}$ and therefore
${\rm Index}(\tilde{\bf I}^\gamma_{p,q})={\rm Index}({\bf I}^\gamma_{p,q})=0$.
Otherwise, Proposition~\ref{prop:Sa2.9} leads to a contradiction because $\gamma$ is a minimizer of the Kropina length functional on $\mathscr{C}$.

Similarly, the closed geodesic obtained by \cite[Theorem~5.1]{CaGiMaSu} contains no conjugate points and has index zero.

\section*{Acknowledgments}
 The author is deeply grateful to
 the anonymous referee for some interesting questions,  numerous comments
 and improved suggestions.


\renewcommand{\refname}{REFERENCES}

\medskip

\begin{tabular}{l}
 School of Mathematical Sciences, Beijing Normal University\\
 Laboratory of Mathematics and Complex Systems, Ministry of Education\\
 Beijing 100875, The People's Republic of China\\
 E-mail address: gclu@bnu.edu.cn\\
\end{tabular}

\end{document}